\documentclass[11pt, a4paper]{article}

\usepackage[english]{babel}
\usepackage[a4paper]{geometry}
\usepackage{amsmath,amssymb}
\usepackage{amsfonts}
\usepackage{amsthm}
\usepackage{graphicx}
\usepackage[utf8]{inputenc}
\usepackage{mathtools}
\usepackage{paralist}
\usepackage{xcolor}
\usepackage{tikz}
\usetikzlibrary{decorations.markings}
\tikzset{
  mid arrow/.style={
    postaction={
      decorate,
      decoration={
        markings,
        mark=at position 0.5 with {\arrow{stealth}}
      }
    }
  },
  mid arrow reversed/.style={
    postaction={
      decorate,
      decoration={
        markings,
        mark=at position 0.5 with {\arrowreversed{stealth}}
      }
    }
  }
}

\numberwithin{equation}{section}

\theoremstyle{plain}
\newtheorem{theorem}{Theorem}[section]

\newtheorem{lemma}[theorem]{Lemma}
\newtheorem{corollary}[theorem]{Corollary}

\theoremstyle{definition}

\newtheorem{remark}[theorem]{Remark}
\newtheorem{example}[theorem]{Example}

\newcommand{\C}{\mathbb{C}}
\newcommand{\R}{\mathbb{R}}

\newcommand{\Z}{\mathbb{Z}}

\newcommand{\bD}{\mathbb{D}}

\newcommand{\bH}{\mathbb{H}}

\newcommand{\cK}{\mathcal{K}}

\newcommand{\cO}{\mathcal{O}}

\newcommand{\coloneq}{\mathrel{\mathop:}=}

\newcommand{\conj}[1]{\overline{#1}}
\newcommand{\comp}[1]{#1^{\operatorname{c}}}

\newcommand{\dd}{{\operatorname{d}}}
\newcommand{\ee}{{\operatorname{e}}}
\newcommand{\ii}{{\operatorname{i}}}

\DeclarePairedDelimiter{\abs}{\lvert}{\rvert}

\DeclarePairedDelimiter{\cc}{[}{]}
\DeclarePairedDelimiter{\oc}{]}{]}
\DeclarePairedDelimiter{\co}{[}{[}
\DeclarePairedDelimiter{\oo}{]}{[}

\DeclareMathOperator{\re}{Re}
\DeclareMathOperator{\im}{Im}

\DeclareMathOperator{\capacity}{cap}

\DeclareMathOperator{\wind}{wind}

\DeclareMathOperator{\Eta}{H}
\DeclareMathOperator{\cn}{cn}
\DeclareMathOperator{\dn}{dn}
\DeclareMathOperator{\sn}{sn}
\DeclareMathOperator{\zn}{zn}

\newcommand{\iK}{\ii\cK'}

\title{The Complex Green's Function for Symmetric Sets}
\author{Klaus Schiefermayr\footnotemark[1] \and Olivier S\`{e}te\footnotemark[2]}
\date{July 31, 2026}

\begin{document}
\maketitle

\renewcommand{\thefootnote}{\fnsymbol{footnote}}

\footnotetext[1]{University of Applied Sciences Upper Austria, Campus Wels, 
Austria, \\ \texttt{klaus.schiefermayr@fh-wels.at}}

\footnotetext[2]{University of Ulm, Faculty of Mathematics and Economics, Helmholtzstra{\ss}e 18, 89081 Ulm, Germany. \texttt{olivier.sete@uni-ulm.de}}

\renewcommand{\thefootnote}{\arabic{footnote}}

\begin{abstract}
The aim of this paper is threefold:
(i) We study the complex Green's function (the analytic extension of the real Green's function) for multiply connected domains with some symmetry and transform it to a simple form with the help of Walsh's conformal map onto lemniscatic domains.
(ii) For the complement of the union of two real intervals, we represent the complex Green's function with the help of Jacobi's elliptic and theta functions.
(iii) Using this representation, we explicitly obtain all parameters of the lemniscatic domain corresponding to the complement of the two intervals.
In addition, using an equality between the corresponding complex Green's functions, we obtain a numerical method for computing the conformal map from the complement of two intervals onto a lemniscatic domain which yields more accurate results than a previous method from the literature.
\end{abstract}



\section{Introduction}

Let $E$ be the union of $\ell$ disjoint, simply connected and compact sets $E_1, \ldots, E_\ell \subseteq \C$ with logarithmic capacity $\capacity(E_j) > 0$, $j = 1, \ldots, \ell$, i.e., a set of the form
\begin{equation} \label{eqn:E_general}
E = \bigcup_{j=1}^\ell E_j.
\end{equation}
Let $g_E : \C \setminus E \to \oo{0, +\infty}$ denote the (real) Green's function with pole at infinity of $\comp{E} \coloneq \widehat{\C} \setminus E$, where $\widehat{\C} = \C \cup \{ \infty \}$ denotes the extended complex plane.  Since the Green's function $g_E$ is harmonic, it has a conjugate harmonic function $h_E$ in $\C \setminus E$, which is unique up to an additive constant and is in general multi-valued.  Hence,
\begin{equation}
G_E(z) \coloneq g_E(z) + \ii h_E(z), \quad z \in \C \setminus E,
\end{equation}
is a multi-valued analytic function with $g_E = \re(G_E)$.
We call $G_E$ a \emph{complex Green's function} of $\comp{E}$.
The function $G_E$ and functions related to $G_E$ frequently appear in the 
literature, but naming seems not standardized.
Fischer~\cite[p.~93]{Fischer1996}, following 
Akhiezer~\cite[p.~1178]{Akhiezer1932}, calls $\exp(-G_E)$ complex Green's 
function,
Christiansen, Simon and Zinchenko~\cite{ChristiansenSimonZinchenko2025}
call $\exp(-G_E)$ the complexified exponential variant of the Green's function,
Walsh~\cite[p.~65]{Walsh1969} considers $\exp(+G_E)$ without giving it a 
special name,
Widom~\cite[p.~140]{Widom1969} considers $G_E$ without giving it a name and 
calls $\exp(+G_E)$ exponential Green's function~\cite[pp.~167, 208, 
219]{Widom1969},
Needham~\cite[p.~551]{Needham1997} calls $-G_E$ a complex potential.

An advantage of the complex Green's function compared to the real Green's function is that the imaginary part of $G_E$ contains additional information, in particular on $\partial E$.
For instance, if $E$ is the unit disk, $E = \{ z \in \C : \abs{z} \leq 1 \}$, the real and complex Green's function of $\comp{E}$ are $g_E(z) = \log \abs{z}$ and $G_E(z) = \log(z) = \log \abs{z} + \ii \arg(z)$.
Thus, on the unit circle $\partial E$, $g_E(z) = 0$ whereas $G_E(z) = \ii \arg(z)$.
Here, the function $G_E$ can be made single-valued and unique by 
restricting $G_E$ on $\C \setminus (E \cup \oc{-\infty, -1})$ 
and $G_E(1) = 0$.

The main results of this paper are as follows.

In Section~\ref{sect:complex_Greens_function}, we consider a general set $E$ with symmetric components $E_j$ and investigate the complex Green's function $G_E$.
We choose a single-valued branch of $G_E$ with a branch cut along the real line, and explicitly determine the boundary values of $G_E$ on $\R \setminus E$ in terms of the equilibrium measure of $E$ evaluated at the components $E_1, \ldots, E_\ell$; see Theorem~\ref{thm:GE}.
For the union of $\ell$ disjoint real intervals, we determine the boundary values of $G_E$ also on $E$; see Corollary~\ref{cor:GE_ell_intervals}.

In Section~\ref{sect:walsh}, we consider Walsh's conformal map $\Phi : \comp{E} \to \comp{L}$ onto a lemniscatic domain $\comp{L}$.  The Green's function of this standard domain is of the simplest possible form.  We investigate the corresponding complex Green's function $G_L$ and explicitly determine branches of $G_E$ and $G_L$ such that $G_E(z) = G_L(\Phi(z))$ holds; see Theorem~\ref{thm:GL}.

In Section~\ref{sect:elliptic_characterization} and Section~\ref{sect:conformal_map}, we consider the case that $E$ is the union of two real intervals.
The main results of Section~\ref{sect:elliptic_characterization} are
(i)~the representation of the complex Green's function $G_E$ in terms of Jacobi's elliptic and theta functions, see Theorem~\ref{thm:GE_two_intervals}, and
(ii)~the explicit determination of the parameters of the lemniscatic domain; see Theorem~\ref{thm:lemniscatic_domain_elliptic}.
In Section~\ref{sect:conformal_map}, we briefly discuss the numerical computation of Walsh's conformal map $\Phi : \comp{E} \to \comp{L}$.
Our numerical examples suggest that using the elliptic representation yields more accurate results than using the representation with improper integrals as in~\cite{SchiefermayrSete2025}.

Finally, in Section~\ref{sect:elliptic_functions}, the definitions and some well-known results of Jacobi's elliptic and theta functions are provided.  Moreover, we prove a new result on a certain quotient of Jacobi's theta functions, see Theorem~\ref{thm:Eta_quotient}, which is crucial for the representation of the complex Green's function in Theorem~\ref{thm:GE_two_intervals}.

\section{The complex Green's function for symmetric sets}
\label{sect:complex_Greens_function}

Let $E = \cup_{j=1}^\ell E_j$ be as in~\eqref{eqn:E_general}, let $g_E$ and $G_E$ be the real and complex Green's function of $\comp{E}$, respectively.  As mentioned in the introduction, $G_E$ is unique up to an additive (imaginary) constant and $G_E$ is a multi-valued analytic function. In contrast, since
\begin{equation} \label{eqn:GEprime}
2 \partial_z g_E(z) = G_E'(z), \quad z \in \C \setminus E,
\end{equation}
the derivative of $G_E$ is a single-valued analytic function in $\C \setminus E$.
Here, $\partial_z = \frac{1}{2} (\partial_x - \ii \partial_y)$ denotes the Wirtinger derivative.
In particular, if $f$ is analytic, then
\begin{equation} \label{eqn:wirtinger}
2 \partial_z \re(f(z)) = f'(z)
\quad \text{and} \quad
2 \partial_z \log \abs{f(z)} = \frac{f'(z)}{f(z)} = \frac{\dd}{\dd z} \log(f(z)).
\end{equation}

Let $\mu_E$ denote the equilibrium measure of $E$. For $j \in \{ 1, \ldots, \ell \}$, let $\gamma_j$ be a smooth loop in $\C \setminus E$  with $\wind(\gamma_j; z) = 1$ for $z \in E_j$ and $\wind(\gamma_j; z) = 0$ for $z \in E_k$, $k \neq j$. Then, by~\cite[Sect.~2]{SchiefermayrSete2025},
\begin{equation} \label{eqn:mj_and_GE}
\mu_E(E_j)
= \frac{1}{2 \pi \ii} \int_{\gamma_j} 2 \partial_z g_E(z) \, \dd z
= \frac{1}{2 \pi \ii} \int_{\gamma_j} G_E'(z) \, \dd z.
\end{equation}
Since $G_E(z) = G_E(z_0) + \int_{z_0}^z G_E'(\zeta) \, \dd \zeta$, it follows from~\eqref{eqn:mj_and_GE}  that two values of $G_E(z)$ differ by a term $2\pi\ii\sum_{j=1}^\ell \nu_j \mu_E(E_j)$ with $\nu_1, \ldots, \nu_\ell \in \Z$.

In the following, we consider the case that all components of $E$ are symmetric in the sense that
\begin{equation}
E_j^* = E_j, \quad j = 1, \ldots, \ell,
\end{equation}
where $S^* \coloneq \{ \conj{z} : z \in S \}$ denotes the reflection of 
a set $S \subseteq \C$ with respect to the real line.  Then, 
by~\cite[Lem.~A.2]{SchiefermayrSete2023}, the set $E_j \cap \R$ is either an 
interval or a single point, i.e.,
\begin{equation}
E_j \cap \R = \cc{b_{2j-1}, b_{2j}}
\end{equation}
with $b_{2j-1} \leq b_{2j}$ for $j = 1, \ldots, \ell$.  We label the components 
$E_1, \ldots, E_\ell$ such that
\begin{equation}
b_1 \leq b_2 < b_3 \leq b_4 < \ldots < b_{2 \ell - 1} \leq b_{2 \ell}.
\end{equation}
Moreover, $\R \setminus E$ consists of $\ell+1$ real intervals, denoted by
\begin{equation} \label{eqn:intervals_Ij}
I_0 \coloneq \oo{-\infty, b_1}, \quad
I_j \coloneq \oo{b_{2j}, b_{2j+1}}, \quad j = 1, \ldots, \ell-1, \quad
I_\ell \coloneq \oo{b_{2 \ell}, +\infty}.
\end{equation}
In order to obtain a single-valued branch of $G_E$, we restrict $G_E$ to the simply 
connected domain $\C \setminus (E \cup \co{b_1, \infty})$.
Under the assumption that $\partial E$ is \emph{regular} in the sense that 
the Green's function $g_E$ has a continuous extension to $\C$ with $g_E(z) = 0$ for 
$z \in E$, we determine the values of $G_E$ on $\R \setminus E$ and at the 
points $b_1, \ldots, b_{2 \ell}$, where that branch of $G_E$ with $G_E(b_1) = 0$ 
is selected.
Let us denote the upper and the lower half-plane by
\begin{equation}
\bH^+ \coloneq \{ z \in \C : \im(z) > 0 \} \quad \text{and} \quad
\bH^- \coloneq \{ z \in \C : \im(z) < 0 \}, 
\end{equation}
respectively.


\begin{theorem} \label{thm:GE}
Let $E = \cup_{j=1}^\ell E_j$ be as in~\eqref{eqn:E_general} with $E_j^* = E_j$ for $j = 1, \ldots, \ell$, and with regular boundary.  The complex Green's function $G_E$ has the following properties.
\begin{enumerate}
\item \label{it:GE_symmetry}
We have $G_E'(\conj{z}) = \conj{G_E'(z)}$, $z \in \C \setminus E$.  In particular, $G_E'(z)$ is real for $z \in \R \setminus E$.

\item \label{it:GE_branch}
Restricting $G_E$ to the simply connected domain $\C \setminus (E \cup \co{b_1, \infty})$, there exists a single-valued branch of $G_E$ with
\begin{equation}
G_E(b_1) \coloneq \lim_{\substack{x \to b_1 \\ x < b_1}} G_E(x) = 0.
\end{equation}
This branch has real values on $\oo{-\infty, b_1}$, and, consequently, $G_E(z) = \conj{G_E(\conj{z})}$ for $z \in \C \setminus (E \cup \co{b_1, \infty})$.

\item \label{it:GE_values_on_Ij}
The branch of $G_E$ in~\ref{it:GE_branch} has the following boundary values on 
the intervals $I_j$:
\begin{equation} \label{eqn:GE_on_Ij}
G_E^\pm(x) \coloneq
\lim_{\substack{z \to x \\ z \in \bH^\pm \setminus E}} G_E(z) = g_E(x) \mp \ii \pi \sum_{k=1}^j \mu_E(E_k), \quad x \in I_j, \quad j = 
0, \ldots, \ell.
\end{equation}
In particular, for $j = 1, \ldots, \ell-1$,
\begin{equation} \label{eqn:GE_at_bj}
G_E^\pm(b_{2j}) = G_E^\pm(b_{2j+1}) = \mp \ii \pi \sum_{k=1}^j \mu_E(E_k),
\quad
G_E^\pm(b_{2 \ell}) = \mp \ii \pi.
\end{equation}
\end{enumerate}
\end{theorem}

\begin{proof}
\ref{it:GE_symmetry}
By~\cite[Lem.~A.1]{SchiefermayrSete2023}, if $E^* = E$ then Green's function satisfies $\partial_z g_E(\conj{z}) = \conj{\partial_z g_E(z)}$ for $z \in \C 
\setminus E$.
By~\eqref{eqn:GEprime}, we obtain~\ref{it:GE_symmetry}.

\ref{it:GE_branch}
Since $\C \setminus (E \cup \co{b_1, \infty})$ is simply connected, $g_E$ has a 
single-valued harmonic conjugate $h_E$ (unique up to an additive real constant) and $G_E = g_E + \ii h_E$ is 
single-valued in $\C \setminus (E \cup \co{b_1, \infty})$.
Fix $x_0 \in \oo{-\infty, b_1}$, then, for any $x < b_1$,
\begin{equation}
G_E(x) = G_E(x_0) + \int_{x_0}^x G_E'(\zeta) \, \dd \zeta,
\end{equation}
where $G_E'(\zeta)$ is real (for real $\zeta$).
Taking the imaginary part yields
$h_E(x) = h_E(x_0)$, i.e., $h_E$ is constant on $\oo{-\infty, b_1}$.
In particular,
$h_E(b_1) \coloneq \lim_{x \to b_1, x < b_1} h_E(x) = h_E(x_0)$.
Since $h_E$ is unique up to an additive real constant, we choose $h_E$ such 
that $h_E(x) = 0$ on $\oc{-\infty, b_1}$, in particular $h_E(b_1) = 0$.
By assumption, $b_1 \in \partial E$ is a regular boundary point, hence also 
$g_E(b_1) = 0$.  This shows~\ref{it:GE_branch}.

\ref{it:GE_values_on_Ij}
By the proof of~\ref{it:GE_branch}, we have $G_E(x) = g_E(x)$ for $x \in I_0$, which gives~\eqref{eqn:GE_on_Ij} for $j = 0$.
Next, let $\gamma_1$ be a smooth path from $b_1$ to $x \in I_1$ in $\bH^- \setminus E$.  Then, $\gamma = \gamma_1 + \gamma_2$ with $\gamma_2 = - \conj{\gamma}_1$
is a loop in $\C \setminus E$ with $\wind(\gamma, E_k) = \delta_{k,1}$.
Hence, by~\eqref{eqn:mj_and_GE} and~\cite[Lem.~A.3]{SchiefermayrSete2023} together with~\ref{it:GE_symmetry},
\begin{align*}
\pi \mu_E(E_1)
&= \frac{1}{2 \ii} \int_\gamma G_E'(\zeta) \, \dd \zeta
= \frac{1}{2 \ii} \int_{\gamma_1} G_E'(\zeta) \, \dd \zeta
- \frac{1}{2 \ii} \int_{\conj{\gamma}_1} G_E'(\zeta) \, \dd \zeta \\
&= \frac{1}{2 \ii} \int_{\gamma_1} G_E'(\zeta) \, \dd \zeta
- \frac{1}{2 \ii} \conj{\int_{\gamma_1} G_E'(\zeta) \, \dd \zeta}
= \im \left( \int_{\gamma_1} G_E'(\zeta) \, \dd \zeta \right) \\
&= \im (G_E^-(x) - G_E(b_1)) = \im(G_E^-(x)).
\end{align*}
Therefore, $G_E^-(x) = g_E(x) + \ii \pi \mu_E(E_1)$ for $x \in I_1$.
Integrating the other way round gives $G_E^+(x) = g_E(x) - \ii \pi \mu_E(E_1)$ 
for $x \in I_1$.  This shows~\eqref{eqn:GE_on_Ij} and~\eqref{eqn:GE_at_bj} for 
$j = 1$.
For $x \in I_j$, $j = 2, \ldots, \ell$, we consider a path $\gamma_1$ from 
$b_1$ to $x \in I_j$ in $\bH^- \setminus E$, and proceed as above.
\end{proof}

Let $E$ be the union of $\ell$ real intervals, i.e.,
\begin{equation} \label{eqn:E_ell_intervals}
E = \bigcup_{j=1}^\ell \cc{b_{2j-1}, b_{2j}}, \quad
b_1 < b_2 < \ldots < b_{2 \ell}.
\end{equation}
Then, the Green's function of $\comp{E}$ is
\begin{equation} \label{eqn:Greens_function_ell_intervals}
g_E(z) = \re \left( \int_{b_1}^z \frac{R(\zeta)}{\sqrt{H(\zeta)}} \, \dd \zeta \right), \quad z \in \C \setminus E,
\end{equation}
where $H(z) \coloneq \prod_{j=1}^{2 \ell} (z - b_j)$, 
the branch of the square root is chosen such that $\sqrt{H(z)}$ behaves as $z^\ell$ at $\infty$,
and $R$ is a certain real polynomial; see, e.g.,~\cite[Thm.~4.2]{SchiefermayrSete2025}.
The integration is along any path in $\C \setminus E$ from $b_1$ to $z$.
Therefore, the complex Green's function of $\comp{E}$ is
\begin{equation} \label{eqn:complex_Greens_function_ell_intervals}
G_E(z) = \int_{b_1}^z \frac{R(\zeta)}{\sqrt{H(\zeta)}} \, \dd \zeta, \quad z \in \C \setminus E,
\end{equation}
which is a multi-valued analytic function in $\C \setminus E$.
As in Theorem~\ref{thm:GE}, restricting $G_E$ to the simply connected domain $\C \setminus \co{b_1, \infty}$ yields a single-valued branch with $G_E(b_1) = 0$.
For this branch, the integration in~\eqref{eqn:complex_Greens_function_ell_intervals} is along any path in $\C \setminus \co{b_1, \infty}$ from $b_1$ to $z$.
In this notation, the equilibrium measure of $E$ satisfies
\begin{equation} \label{eqn:equilibrium_measure}
\dd \mu_E(x) = \frac{1}{\pi} \frac{\abs{R(x)}}{\sqrt{\abs{H(x)}}} \, \dd x, \quad x \in E,
\end{equation}
with the positive real square root; see~\cite[Thm.~4.2\,(iii)]{SchiefermayrSete2025}.
The boundary values of the branch of $G_E$ in Theorem~\ref{thm:GE}\,\ref{it:GE_branch} on $E$ can be computed explicitly.
Note that $E$ in~\eqref{eqn:E_ell_intervals} has regular boundary.

\begin{corollary} \label{cor:GE_ell_intervals}
Let $E$ be as in~\eqref{eqn:E_ell_intervals} and consider the branch of $G_E$ as above.  Then,
\begin{equation} \label{eqn:GEpm_on_E}
G_E^\pm(x) \coloneq \lim_{\substack{z \to x \\ z \in \bH^\pm}} G_E(z)
= \mp \ii \pi \mu_E(\cc{b_1, x}), \quad x \in E,
\end{equation}
or, equivalently, since the support of $\mu_E$ is contained in $E$,
\begin{equation}
G_E^\pm(x) = \mp \ii \pi \sum_{k=1}^{j-1} \mu_E(\cc{b_{2k-1}, b_{2k}}) \mp \ii \pi \mu_E(\cc{b_{2j-1}, x}),
\quad x \in \cc{b_{2j-1}, b_{2j}},
\end{equation}
for $j = 1, \ldots, \ell$.
In particular, the density of the equilibrium measure of $E$ can be represented by
\begin{equation} \label{eqn:density_muE}
\frac{\dd \mu_E(x)}{\dd x} = \frac{1}{\pi \ii} \frac{\dd G_E^-}{\dd x}(x)
= - \frac{1}{\pi \ii} \frac{\dd G_E^+}{\dd x}(x), 
\quad x \in E.
\end{equation}
\end{corollary}

\begin{proof}
We directly evaluate $G_E^\pm(x)$ by integrating along the real line (seen as the boundary of $\bH^\pm$).
The values of $R$ and $\sqrt{H}$ on the real line are given in~\cite[Lem.~4.1 and Cor.~4.3]{SchiefermayrSete2025}, which we use in the following.
For $x \in \cc{b_{2j-1}, b_{2j}}$,
\begin{equation*}
\lim_{\substack{z \to x \\ z \in \bH^\pm}} \frac{R(z)}{\sqrt{H(z)}}
= \frac{(-1)^{\ell-j} \abs{R(x)}}{\pm \ii (-1)^{\ell-j} \sqrt{\abs{H(x)}}}
= \mp \ii \frac{\abs{R(x)}}{\sqrt{\abs{H(x)}}},
\end{equation*}
and for $x \in I_j$, see~\eqref{eqn:intervals_Ij}, the function $R(x) / \sqrt{H(x)}$ is real.
Hence, integration over $\cc{b_1, x} \cap E$ contributes only to the imaginary 
part of $G_E^\pm(x)$, and integration over $\cc{b_1, x} \cap (\R \setminus E)$ 
contributes only to the real part of $G_E^\pm(x)$, and thus gives $g_E(x)$.
Moreover, we have $\int_{b_{2j-1}}^{b_{2j}} 
\frac{\abs{R(t)}}{\sqrt{\abs{H(t)}}} \, \dd t = \pi \mu_E(\cc{b_{2j-1}, b_{2j}})$ 
by~\eqref{eqn:equilibrium_measure}, and $\int_{b_{2j}}^{b_{2j+1}} 
\frac{R(t)}{\sqrt{H(t)}} \, \dd t = g_E(b_{2j+1}) - g_E(b_{2j}) = 0$.
Therefore, for $x \in \cc{b_{2j-1}, b_{2j}}$, we obtain
\begin{align*}
G_E^\pm(x) &= \int_{b_1}^x \frac{R(t)}{\sqrt{H(t \pm \ii 0)}} \, \dd t
= \mp \ii \sum_{k=1}^{j-1} \int_{b_{2k-1}}^{b_{2k}} \frac{\abs{R(t)}}{\sqrt{\abs{H(t)}}} \, \dd t \mp \ii \int_{b_{2j-1}}^x \frac{\abs{R(t)}}{\sqrt{\abs{H(t)}}} \, \dd t \\
&= \mp \ii \pi \sum_{k=1}^{j-1} \mu_E(\cc{b_{2k-1}, b_{2k}}) \mp \ii \pi \mu_E(\cc{b_{2j-1}, x})
= \mp \ii \pi \mu_E(\cc{b_1, x}),
\end{align*}
where we used~\eqref{eqn:equilibrium_measure}.
This implies~\eqref{eqn:density_muE}.
\end{proof}

\section{Walsh's conformal map}\label{sect:walsh}

Let $E$ be a compact set as in~\eqref{eqn:E_general}, so that $\comp{E}$ is an $\ell$-connected domain. If $\Phi : \comp{E} \to \comp{L}$ is a conformal map with $\Phi(\infty) = \infty$, then the Green's functions of $\comp{E}$ and $\comp{L}$ satisfy $g_E(z) = g_L(\Phi(z))$; see~\cite[Thm.~4.4.4]{Ransford1995}. Naturally, one is interested in a domain $\comp{L}$ for which $g_L$ has the simplest possible form. Such a domain has been introduced by J.\,L.\,Walsh in his seminal paper~\cite{Walsh1956}, for which the Green's function $g_L$ is just a linear combination of logarithms; see~\eqref{eqn:lemniscatic_domain} and~\eqref{eqn:gL}.

\begin{theorem}[{\cite[Sect.~III]{Walsh1956}}] \label{thm:walsh_map}
Let $E$ be as in~\eqref{eqn:E_general}.
Then, there exists a unique lemniscatic domain $\comp{L} \coloneq \widehat{\C} \setminus L$ with
\begin{equation} \label{eqn:lemniscatic_domain}
L \coloneq \Bigl\{ w \in \C : \prod_{j=1}^\ell \abs{w-a_j}^{m_j} \leq \capacity(E) \Bigr\},
\end{equation}
where $a_1,\ldots,a_\ell\in\C$ are pairwise distinct and
$m_1, \ldots, m_\ell > 0$ satisfy $\sum_{j=1}^\ell m_j = 1$,
and a unique conformal map $\Phi : \comp{E} \to \comp{L}$ with
$\Phi(z) = z + \cO(1/z)$ at $\infty$.
\end{theorem}

The first proof of Theorem~\ref{thm:walsh_map} was given by Walsh~\cite{Walsh1956}, other proofs have been derived by
Grunsky~\cite{Grunsky1957a,Grunsky1957b,Grunsky1978}, 
Jenkins~\cite{Jenkins1958}, and Landau~\cite{Landau1961}.
Both analytic and numerical examples for the conformal map $\Phi$ have been considered in~\cite{SeteLiesen2016}, \cite{SeteLiesen2017}, \cite{NasserLiesenSete2016}, \cite{LiesenSeteNasser2017}, \cite{SchiefermayrSete2023}, \cite{SchiefermayrSete2024}, and~\cite{SchiefermayrSete2025}.
A possible application in approximation theory is the construction of the 
Faber-Walsh polynomials generalizing the Faber polynomials to disconnected 
compact sets; see~\cite{Walsh1958}, \cite{Suetin1998}, 
and~\cite{SeteLiesen2017}.

\begin{remark}
\begin{enumerate}
\item Given the set $E$ as in~\eqref{eqn:E_general}, the centers $a_1, \ldots, a_\ell$ and the exponents $m_1, \ldots, m_\ell$ in Theorem~\ref{thm:walsh_map} are uniquely determined.
In~\cite[Thm.~2.1]{SchiefermayrSete2025}, we proved that the exponents are related to the equilibrium measure $\mu_E$ of $E$ by
\begin{equation} \label{eqn:mj_muE}
m_j = \mu_E(E_j), \quad j = 1, \ldots, \ell.
\end{equation}
Just like the set $E$, the set $L$ in~\eqref{eqn:lemniscatic_domain} consists of $\ell$ disjoint simply connected and compact sets $L_1, \ldots, L_\ell$, i.e., $L = \cup_{j=1}^\ell L_j$, where $a_j \in L_j$, $j = 1, \ldots, \ell$.

\item Let $g_E$ and $g_L$ denote the Green's functions of $\comp{E}$ and $\comp{L}$.  One reason of the importance of Walsh's map is the fact that $g_L$ is of the simplest possible form, i.e.,
\begin{equation} \label{eqn:gL}
g_L(w) = \sum_{j=1}^\ell m_j \log \abs{w - a_j} - \log(\capacity(E)), \quad w \in \C \setminus L.
\end{equation}
By~\cite[Thm.~4.4.4]{Ransford1995}, see also~\cite[p.~28]{Walsh1958},
\begin{equation} \label{eqn:gE_gL}
g_E(z) = g_L(\Phi(z)), \quad z \in \C \setminus E,
\end{equation}
since $\Phi : \comp{E} \to \comp{L}$ is conformal with $\Phi(\infty) = \infty$.
By~\eqref{eqn:gL}, a complex Green's function of $\comp{L}$ is
\begin{equation} \label{eqn:GL}
G_L(w) = \sum_{j=1}^\ell m_j \log (w - a_j) - \log(\capacity(E)), \quad w \in \C \setminus L.
\end{equation}
\end{enumerate}
\end{remark}

From now on, we consider sets $E$ with components symmetric with respect to the real line, that is, $E_j^* = E_j$, $j = 1, \ldots, \ell$.
Then, the same holds for $L$, i.e., $L = 
\cup_{j=1}^\ell L_j$, and $L_j^* = L_j$, $j = 1, \ldots, \ell$, and the centers $a_1, \ldots, a_\ell$ are real,
which follows by combining~\cite[Lem.~2.2]{SeteLiesen2016} 
and~\cite[Thm.~2.7]{SchiefermayrSete2023}.
In this case,
\begin{equation} \label{eqn:real_parts_of_L}
L_j \cap \R = \cc{c_{2j-1}, c_{2j}}, \quad j = 1, \ldots, \ell,
\end{equation}
with $c_1 < c_2 < \ldots < c_{2 \ell}$ and $c_{2j-1} < a_j < c_{2j}$ for $j = 
1, \ldots, \ell$.  Let us denote
\begin{equation} \label{eqn:gaps_in_lemniscatic_domain}
J_0 \coloneq \oo{-\infty, c_1}, \quad
J_j \coloneq \oo{c_{2j}, c_{2j+1}}, \quad j = 1, \ldots, \ell-1, \quad
J_\ell \coloneq \oo{c_{2 \ell}, \infty}.
\end{equation}

In the next theorem, we select a suitable branch of $G_L$, show that $G_L(\Phi(z)) = G_E(z)$, and determine the boundary values of $G_L$.

\begin{theorem} \label{thm:GL}
If all components of $E$ are symmetric with respect to $\R$ then, in the above notation, $G_L$ has the following properties.
\begin{enumerate}
\item \label{it:GL_branch}
Restricting $G_L$ to the simply connected domain $\C \setminus (L \cup \co{c_1, 
\infty})$, there exists a single-valued branch of $G_L$ with $G_L(c_1) = 0$.
This branch is given by
\begin{equation} \label{eqn:GL_branch}
G_L(w) = \sum_{j=1}^\ell m_j \log (a_j - w) - \log(\capacity(E))
= g_L(w) + \ii \sum_{j=1}^\ell m_j \arg (a_j - w)
\end{equation}
with the principal branch of the logarithm and argument, i.e., $\arg : \C 
\setminus \oc{-\infty, 0} \to \oo{-\pi, \pi}$.
This branch of $G_L$ has real values on $\oc{-\infty, c_1}$ and, consequently, 
$G_L(w) = \conj{G_L(\conj{w})}$ for $w \in \C \setminus (L \cup \co{c_1, 
\infty})$.

\item \label{it:GL_GE_relation}
The branch of $G_L$ in~\ref{it:GL_branch} satisfies
\begin{equation}
G_L(\Phi(z)) = G_E(z), \quad z \in \C \setminus (E \cup \co{b_1, \infty}),
\end{equation}
with the branch of $G_E$ from 
Theorem~\ref{thm:GE}\,\ref{it:GE_branch}.

\item \label{it:GL_boundary_values}
The branch of $G_L$ in~\ref{it:GL_branch} in $\C \setminus (L \cup \co{c_1, 
\infty})$ has the following boundary values:
\begin{align}
G_L(w) &= \ii \sum_{j=1}^\ell m_j \arg(a_j - w), \quad w \in \partial L \setminus \oo{c_1, \infty}, \label{eqn:GL_bdry_values_notgaps} \\
G_L^\pm(u) &\coloneq \lim_{\substack{w \to u \\ w \in \bH^\pm \setminus L}} G_L(w)
= g_L(u) \mp \ii \pi \sum_{k=1}^j m_k, \quad u \in \overline{J}_j, \quad j = 0, 1, \ldots, \ell.
\label{eqn:GL_bdry_values_gaps}
\end{align}
In particular,
\begin{equation}
G_L(c_1) = 0, \quad
G_L^\pm(c_{2j}) = G_L^\pm(c_{2j+1}) = \mp \ii \pi \sum_{k=1}^j m_k, \quad
G_L^\pm(c_{2 \ell}) = \mp \ii \pi.
\end{equation}
\end{enumerate}
\end{theorem}

\begin{proof}
The complex Green's function $G_L$ of $\comp{L}$ has the form~\eqref{eqn:GL}.

\ref{it:GL_branch}
Formula~\eqref{eqn:GL_branch} defines a single-valued branch of $G_L$ in $\C \setminus (L \cup \co{c_1, \infty})$.
Since $w < c_1$ implies $w < a_j$ for $j = 1, \ldots, \ell$, this branch is real-valued on $\oc{-\infty, c_1}$ with $G_L(c_1) = g_L(c_1) = 0$, since $c_1 \in \partial L$.

\ref{it:GL_GE_relation} By~\cite[Thm.~2.1]{SchiefermayrSete2024},
the image of $z \in \C \setminus (E \cup \co{b_1, \infty})$ satisfies
$\Phi(z) \in \C \setminus (L \cup \co{c_1, \infty})$,
and $z < b_1$ implies $\Phi(z) < c_1$.  
By~\eqref{eqn:gE_gL}, $g_L(\Phi(z)) = g_E(z)$.  Since $G_E(z) = g_E(z)$ and $G_L(\Phi(z)) = g_L(\Phi(z))$ for $z < b_1$, we obtain $G_E(z) = G_L(\Phi(z))$ for $z < b_1$, which extends to all $z \in \C \setminus (E \cup \co{b_1, \infty})$ by the identity principle for analytic functions.

\ref{it:GL_boundary_values}
The representation~\eqref{eqn:GL_branch} implies~\eqref{eqn:GL_bdry_values_notgaps}, since $g_L(w) = 0$ for $w \in \partial L$.  It also implies~\eqref{eqn:GL_bdry_values_gaps},
since
\begin{equation}
\lim_{\substack{w \to u \\ w \in \bH^\pm \setminus L}} \arg(a_j - w) =
\begin{cases}
0, & u < a_j \\
\mp \pi, & u > a_j.
\end{cases}
\end{equation}
This completes the proof.
\end{proof}

\begin{remark}
In Theorem~\ref{thm:GE} and Theorem~\ref{thm:GL}, we chose the branch cut to the right.  Of course, other branch cuts are possible, and another natural choice is to consider $G_E$ on $\C \setminus (E \cup \oc{-\infty, b_{2 \ell}})$ with $G_E(b_{2 \ell}) = 0$, which has the boundary values
$G_E^\pm(x) = g_E(x) \pm \ii \pi \sum_{k=j+1}^\ell \mu_E(E_k),$ $x \in I_j$, $j = 0, \ldots, \ell$,
and $G_L$ on $\C \setminus (L \cup \oc{-\infty, c_{2 \ell}})$ with the representation~\eqref{eqn:GL} with the principal branch of the logarithm and boundary values
$G_L(w) = \ii \sum_{j=1}^\ell m_j \arg(w - a_j)$ for $w \in \partial L \setminus \oo{-\infty, c_{2 \ell}}$ and 
$G_L^\pm(u) = g_L(u) \pm \ii \pi \sum_{k=j+1}^\ell m_k$, $u \in J_j$, $j = 0, \ldots, \ell$.
\end{remark}

\section{The complex Green's function for two intervals}
\label{sect:elliptic_characterization}

For given $b_1 < b_2 < b_3 < b_4$, consider the set
\begin{equation} \label{eqn:E}
E \coloneq \cc{b_1, b_2} \cup \cc{b_3, b_4}
\end{equation}
consisting of two disjoint real intervals.
In Theorem~\ref{thm:GE_two_intervals}, we represent the complex Green's functions $G_E$ with the help of Jacobi's theta and elliptic functions.  With this result, in Theorem~\ref{thm:lemniscatic_domain_elliptic}, we explicitly determine the parameters $m_1, m_2$ and $a_1, a_2$ of the lemniscatic domain from Theorem~\ref{thm:walsh_map}.

Define the modulus $k$ by
\begin{equation} \label{eqn:k}
k \coloneq \sqrt{\frac{(b_4-b_1)(b_3-b_2)}{(b_4-b_2)(b_3-b_1)}}.
\end{equation}
Since $(b_4 - b_2)(b_3 - b_1) - (b_4-b_1)(b_3-b_2) = (b_4-b_3)(b_2-b_1) > 0$, we have $0 < k < 1$.  Let $\cK\equiv\cK(k)$ be the corresponding elliptic integral of the first kind, let
\begin{equation} \label{eqn:k'}
k' \coloneq \sqrt{1-k^2} = \sqrt{\frac{(b_4-b_3)(b_2-b_1)}{(b_4-b_2)(b_3-b_1)}}
\end{equation}
be the complementary modulus, and let $\cK' \equiv \cK'(k) \coloneq \cK(k')$ be the corresponding associated elliptic integral of the first kind. Note that $\cK,\cK'>0$. All Jacobian elliptic and theta functions in this section have modulus~$k$.  The definitions and properties of these functions are provided in Section~\ref{sect:elliptic_functions}.

Next, we define a parameter $\lambda \in \oo{0, 1}$ that turns out to be crucial for what follows.  Since $\sn : \cc{0, \cK} \to \cc{0, 1}$ is bijective (Lemma~\ref{lem:sn-bijective}) and $\sqrt{(b_4-b_2)/(b_4-b_1)} \in \oo{0, 1}$, there exists a unique $\lambda \in \oo{0,1}$ such that
\begin{equation} \label{eqn:lambda}
\sn(\varrho) = \sqrt{\frac{b_4-b_2}{b_4-b_1}}, \quad \varrho=\lambda\cK.
\end{equation}
By~\eqref{eqn:k}, \eqref{eqn:lambda}, and \eqref{eqn:sncndn}, we have
\begin{equation} \label{eqn:cndn_rho}
\cn(\varrho) = \sqrt{\frac{b_2-b_1}{b_4-b_1}} \quad\text{and}\quad \dn(\varrho) = \sqrt{\frac{b_2-b_1}{b_3-b_1}}.
\end{equation}


With these preliminaries, we are in position to introduce a conformal map~$\varphi$ from a certain rectangular domain onto $\comp{E}$,
which has been studied by Akhiezer~\cite[p.~1175, 
Eqn.~(5)]{Akhiezer1932} and Fischer~\cite[p.~116]{Fischer1996} (with the 
normalization $b_1 = -1$, $b_4 = 1$).  Note that in~\cite{Akhiezer1932, 
Fischer1996}, $-\varrho$ is used instead of $\varrho$.
For completeness, in the appendix, we provide a proof which differs significantly from that in~\cite{Akhiezer1932, Fischer1996}.
In Figure~\ref{fig:varphi}, we illustrate the maps $u \mapsto \sn(u)$, $u \mapsto \sn^2(u)$, and $u \mapsto \varphi(u)$.

\begin{lemma} \label{lem:varphi}
For given $E = \cc{b_1, b_2} \cup \cc{b_3, b_4}$, let $k$ be as in~\eqref{eqn:k} and $\varrho$ be as in~\eqref{eqn:lambda}.
Then,
\begin{equation} \label{eqn:phi}
\varphi(u) \coloneq \frac{b_2 \sn^2(u) - b_1 \sn^2(\varrho)}{\sn^2(u) - \sn^2(\varrho)}
\end{equation}
is an even elliptic function of order~$2$ with fundamental periods $\{ 2 \cK, 2 \iK \}$ and simple poles at $u = \pm \varrho$.  The maps
\begin{equation} \label{eqn:phi_onto_E_complement}
\varphi : P \coloneq \{ u \in \C : 0 < \re(u) < \cK, - \cK' < \im(u) \leq \cK' \} \to \comp{E}
\end{equation}
and
\begin{equation} \label{eqn:varphi_Pplus}
\varphi : P^\pm \coloneq \bigl\{ t \cK \pm \ii t' \cK' : 0<t,t'<1 \bigr\} \to \bH^\mp,
\end{equation}
are conformal and bijective.  On the boundary, the maps
\begin{equation} \label{eqn:phi_boundary_correspondence}
\begin{aligned}
&\varphi:\co{0,\varrho}\to\oc{-\infty,b_1}, \quad
\varphi:\cc{0,\iK}\to\cc{b_1, b_2}, \quad
\varphi:\cc{\iK,\cK+\iK}\to\cc{b_2,b_3}, \\
&\qquad\varphi:\cc{\cK,\cK+\iK}\to\cc{b_3, b_4}, \quad
\varphi:\oc{\varrho,\cK} \to\co{b_4,+\infty}
\end{aligned}
\end{equation}
are bijective and
\begin{equation}\label{SpecialValueszu}
\varphi(0)=b_1, \quad \varphi(\iK)=b_2, \quad \varphi(\cK+\iK)=b_3, \quad \varphi(\cK)=b_4, \quad \varphi(\varrho)=\infty.
\end{equation}
Moreover, the inverse of $\varphi$ is
\begin{equation} \label{eqn:varphi_inv}
\varphi^{-1} : \comp{E} \to P, \quad
\varphi^{-1}(z) = \sn^{-1} \biggl( \sn(\varrho) \sqrt{\frac{z-b_1}{z-b_2}} 
\biggr),
\end{equation}
with the principal branch of the square root.
\end{lemma}

\begin{figure}[t!]
{\centering
\begin{tikzpicture}[
declare function = {
b1 = -2; b2 = -1; b3 = 0.5; b4 = 3;
k = 0.8660;
K = 2.1565;
Kp = 1.6858;
rho = 1.3053;
opac1 = 0.2;
opac2 = 0.4;
}]


\draw[->] (-0.5,0) -- (3.5,0) node[right]{$\re$};
\draw[->] (0,-2) -- (0,2) node[above]{$\im$};

\draw[fill=blue, opacity=opac2] (0,0) -- (K,0) -- (K,Kp) -- (0,Kp) -- (0,0);
\draw[fill=blue, opacity=opac1] (0, 0) -- (K,0) -- (K,-Kp) -- (0,-Kp) -- (0,0);

\draw[ultra thick] (0,0) -- (K,0) -- (K,Kp) -- (0,Kp) -- (0,0);
\draw[ultra thick] (0, 0) -- (K,0) -- (K,-Kp) -- (0,-Kp) -- (0,0);

\draw[ultra thick, mid arrow, green!60!black] (rho,0) -- (K,0);
\draw[ultra thick, mid arrow, teal!60!black] (K,0) -- (K,Kp);
\draw[ultra thick, mid arrow, cyan!60!black] (K,Kp) --(0,Kp);
\draw[ultra thick, mid arrow, blue!60!black] (0,Kp) -- (0,0);
\draw[ultra thick, mid arrow, violet!80!black] (0,0) -- (rho,0);

\node[above left] at (0,0) {$0$};
\node[above right] at (K,0) {$\cK$};
\node[left] at (0,Kp) {$\ii \cK'$};
\node[right] at (K,Kp) {$\cK + \ii \cK'$};
\node[left] at (0,-Kp) {$- \ii \cK'$};
\node[right] at (K,-Kp) {$\cK - \ii \cK'$};
\draw[ultra thick] (rho,0.1)--(rho,-0.1) node[below]{$\varrho$};

\node at (K/2, Kp/2) {$P^+$};
\node at (K/2, -Kp/2) {$P^-$};


\begin{scope}[shift={(6.5,0)}]
\fill[blue, opacity=opac2] (0, 0) rectangle (5, 2);
\fill[blue, opacity=opac1] (0, 0) rectangle (5, -2);

\draw[->] (-1,0) -- (5,0) node[right]{$\re$};
\draw[->] (0,-2) -- (0,2) node[above]{$\im$};

\draw[->, ultra thick] (0,-2) -- (0,2);
\draw[->, ultra thick] (0,0) -- (5,0);

\draw[ultra thick, mid arrow, green!60!black] (1.4,0) -- (2,0);
\draw[ultra thick, mid arrow, teal!60!black] (2,0) -- (3,0);
\draw[ultra thick, mid arrow, cyan!60!black] (3,0) -- (5,0);
\draw[ultra thick, mid arrow, blue!60!black] (0,2) -- (0,0);
\draw[ultra thick, mid arrow, violet!80!black] (0,0) -- (1.4,0);

\node[above left] at (0,0) {$0$};
\foreach \point/\pointlabel in { {1.4}/{$\sn(\varrho)$}, 2/{$1$}, 
3/{$1/k$} }
\draw[ultra thick] (\point,0.1)--(\point,-0.1) node[below]{\pointlabel};

\node at (4, 1.2) {$\sn(u)$};
\end{scope}


\begin{scope}[shift={(0,-5)}]
\fill[blue, opacity=opac2] (-1.5, 0) rectangle (4.5, 2);
\fill[blue, opacity=opac1] (-1.5, 0) rectangle (4.5, -2);

\draw[-] (-1.5,0) -- (4.5,0) node[right]{$\re$};
\draw[->] (0,-2) -- (0,2) node[above]{$\im$};

\draw[ultra thick, mid arrow, green!60!black] (1.7,0) -- (2,0);
\draw[ultra thick, mid arrow, teal!60!black] (2,0) -- (3.5,0);
\draw[ultra thick, mid arrow, ->, cyan!60!black] (3.5,0) -- (4.5,0);
\draw[ultra thick, mid arrow, blue!60!black] (-1.5,0) -- (0,0);
\draw[ultra thick, mid arrow, violet!80!black] (0,0) -- (1.7,0);

\node[above left] at (0,0) {$0$};
\draw[ultra thick] (2,-0.1)--(2,0.1) node[above]{$1$};
\foreach \point/\pointlabel in { 1.7/{$\sn^2(\varrho)$}, {3.5}/{$1/k^2$} }
\draw[ultra thick] (\point,0.1)--(\point,-0.1) node[below]{\pointlabel};

\node at (3.4, 1.2) {$\sn^2(u)$};
\end{scope}


\begin{scope}[shift={(8,-5)}]
\fill[blue, opacity=opac1] (-2.5, 0) rectangle (3.5, 2);
\fill[blue, opacity=opac2] (-2.5, 0) rectangle (3.5, -2);

\draw[-] (-2.5,0) -- (3.5,0) node[right]{$\re$};
\draw[->] (0,-2) -- (0,2) node[above]{$\im$};

\draw[ultra thick, ->] (-2.5,0) -- (3.5,0); 
\draw[ultra thick, mid arrow, green!60!black] (3.5,0) -- (b4,0);
\draw[ultra thick, mid arrow, teal!60!black] (b4,0) -- (b3,0);
\draw[ultra thick, mid arrow, cyan!60!black] (b3,0) -- (b2,0);
\draw[ultra thick, mid arrow, blue!60!black] (b2,0) -- (b1,0);
\draw[ultra thick, mid arrow, violet!80!black] (b1,0) -- (-2.5,0);

\foreach \point/\pointlabel in { b1/{$b_1$}, b2/{$b_2$}, b3/{$b_3$}, 
b4/{$b_4$} }
\draw[ultra thick] (\point,0.1)--(\point,-0.1) node[below]{\pointlabel};

\node at (2.4, 1.2) {$\varphi(u)$};
\end{scope}

\end{tikzpicture}

}
\caption{Illustration of the map $\varphi$ in Lemma~\ref{lem:varphi}.
Top left: $u$-plane.  Top right: $\sn(u)$.
Bottom left: $\sn^2(u)$.  Bottom right: $\varphi(u)$.}
\label{fig:varphi}
\end{figure}
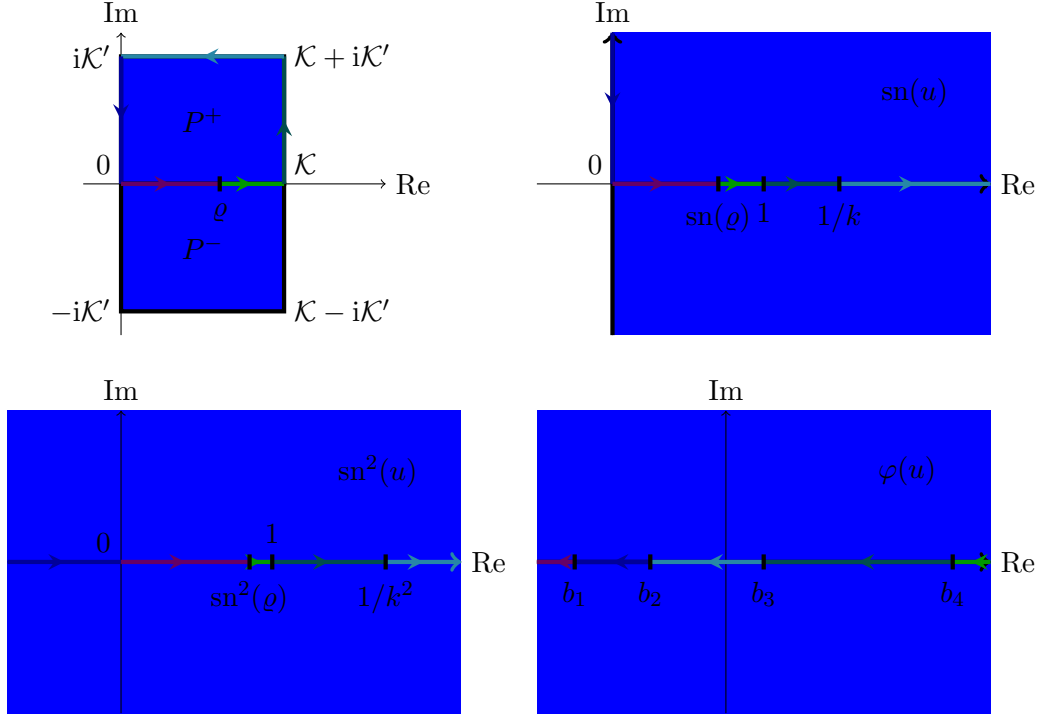

In the next theorem, we recall the representation of the Green's function $g_E$ and the formulas for the logarithmic capacity $\capacity(E)$ and the critical point $z_*$ of $g_E$ with the help of Jacobi's elliptic and theta functions. In the case $b_1 = -1$, $b_4 = 1$, Akhiezer obtained formula~\eqref{eqn:gE} in~\cite[\S{}5]{Akhiezer1932}, see also Fischer~\cite[Thm.~3.3.16]{Fischer1996} for a more recent presentation, \eqref{eqn:capacity_elliptic} in~\cite{Akhiezer1930}, and~\eqref{eqn:z*} in~\cite[p.~1184, (12)]{Akhiezer1932}. For completeness and convenience of the reader, we give a comprehensive and more modern proof in the appendix.


\begin{theorem} \label{thm:Green_function_2_intervals}
For given $E=\cc{b_1, b_2} \cup \cc{b_3, b_4}$, let $k$ be as in~\eqref{eqn:k}, 
let $\varrho=\lambda\cK$ be defined by~\eqref{eqn:lambda}, and let $\varphi$, 
$P$, $P^\pm$ be as in Lemma~\ref{lem:varphi}.
\begin{enumerate}
\item \label{it:green_function}
The Green's function of $\comp{E}$ can be represented by
\begin{equation} \label{eqn:gE}
g_E(z) = \log \left\lvert \frac{\Eta(\varrho + u)}{\Eta(\varrho - u)} \right\rvert, \quad z \in \C \setminus E,
\end{equation}
where $z = \varphi(u)$, $u \in P$.

\item \label{it:logarithmic_capacity}
The logarithmic capacity of $E$ is
\begin{equation} \label{eqn:capacity_elliptic}
\capacity(E)
= \frac{(b_4 - b_1) (b_3 - b_1)}{4 (b_2 - b_1)} \frac{\Theta^4(0)}{\Theta^4(\varrho)}
= \frac{(b_4 - b_1) (b_2 - b_1)}{4 (b_3 - b_1)} \frac{\Theta_1^4(0)}{\Theta_1^4(\varrho)}.
\end{equation}

\item \label{it:critical_point_gE}
Green's function $g_E$ has exactly one critical point given by
\begin{equation}\label{eqn:z*}
z_* = b_2+\sqrt{(b_3-b_1)(b_4-b_2)}\,\zn(\varrho),
\end{equation}
which satisfies $z_* \in \oo{b_2,b_3}$.
\end{enumerate}
\end{theorem}

\begin{remark}
In~\cite[Thm.~3.1]{Schiefermayr2015_Zolotarev}, the logarithmic capacity of a certain set consisting of two Jordan arcs with endpoints $b_1, b_2, b_3, b_4 \in \C$ is obtained using the same technique as in the proof of Theorem~\ref{thm:Green_function_2_intervals}\,\ref{it:logarithmic_capacity}.
\end{remark}

Based on~\eqref{eqn:gE} and Theorem~\ref{thm:Eta_quotient}, we give a representation of the complex Green's function $G_E$ in terms of Jacobi's theta function $\Eta$; see Section~\ref{sect:elliptic_functions}.

\begin{theorem} \label{thm:GE_two_intervals}
In the notation of Theorem~\ref{thm:Green_function_2_intervals},
the complex Green's function $G_E$ can be represented by
\begin{equation} \label{eqn:GE}
G_E(z) = \log \left( \frac{\Eta(\varrho + u)}{\Eta(\varrho - u)} \right), \quad z \in \C \setminus E,
\end{equation}
where $z = \varphi(u)$, $u \in P$.
The single-valued branch of $G_E$ in $\C \setminus \co{b_1, \infty}$ with 
$G_E(b_1) = 0$ from Theorem~\ref{thm:GE} can be represented by~\eqref{eqn:GE} 
with $u \in P^+ \cup P^-\cup \oo{0,\varrho}$ and with the principal branch of 
the logarithm.  Then, $\log \Big( \frac{\Eta(\varrho+u)}{\Eta(\varrho-u)} \Big)$ is real for 
$u \in \oo{0, \varrho}$ with
\begin{equation}
G_E(b_1) = \lim_{\substack{u \to 0 \\ u \in P^+ \cup P^-\cup \oo{0,\varrho}}} 
\log \left( \frac{\Eta(\varrho + u)}{\Eta(\varrho - u)} \right) = 0,
\end{equation}
and
\begin{equation}
G_E^\pm(b_2) = G_E^\pm(b_3) = \mp \ii \pi (1 - \lambda)
\quad \text{and} \quad
G_E^\pm(b_4) = \mp \ii \pi.
\end{equation}
\end{theorem}

\begin{proof}
Formula~\eqref{eqn:GE} follows from~\eqref{eqn:gE}.
The assertions about the branch follow from Theorem~\ref{thm:GE} and 
Theorem~\ref{thm:Eta_quotient}\,\ref{it:f_on_domain} in combination with the 
mapping properties of $\varphi$ in Lemma~\ref{lem:varphi}.
Since $G_E^+(b_2) = \conj{G_E^-(b_2)}$, we only have to consider $G_E^-(b_2)$.
By Lemma~\ref{lem:varphi},
\begin{equation} \label{eqn:GEb2minus}
G_E^-(b_2)
= \lim_{\substack{z \to b_2 \\ z \in \bH^-}} G_E(z)
= \lim_{\substack{u \to \ii \cK' \\ u \in P^+}} G_E(\varphi(u))
= \log \left( \frac{\Eta(\varrho + \ii \cK')}{\Eta(\varrho - \ii \cK')} \right).
\end{equation}
By~\eqref{Eta(uiK)},
\begin{equation*}
\frac{\Eta(\varrho + \ii \cK')}{\Eta(\varrho - \ii \cK')}
= - \ee^{- \ii \pi \varrho / \cK}
= - \ee^{- \ii \pi \lambda}
= \ee^{\ii \pi (1 - \lambda)},
\end{equation*}
hence $G_E^-(b_2) = \ii \pi (1 - \lambda)$.
The values $G_E^\pm(b_3)$ and $G_E^\pm(b_4)$ are obtained with 
Theorem~\ref{thm:GE}.
\end{proof}


\begin{remark}
\begin{enumerate}
\item Figure~\ref{fig:GE} displays phase plots of $f(u) = \Eta(\varrho+u)/\Eta(\varrho-u)$ 
and of the complex Green's function $G_E(\varphi(u))$, $u \in P$, from 
Theorem~\ref{thm:GE_two_intervals}.
In a phase plot of a complex function $f$, the domain is colored
according to its phase $f/\abs{f}$; see~\cite{Wegert2012, 
WegertSemmler2011}.
In Figure~\ref{fig:GE}, the scale of the $y$-axis is roughly half that of the 
$x$-axis ($\cK = 1.9080$, $\cK' = 1.8067$ rounded to four digits).
In the left panel, we see the simple pole of $f$ at $\varrho$ and that 
$f(u) > 0$ on $\co{0, \varrho}$ and $f(u) < 0$ on $\oc{\varrho, \cK}$; compare 
Theorem~\ref{thm:Eta_quotient}\,\ref{it:f_on_boundary}.
In the right panel, we see the zero of $G_E(\varphi(u))$ at $u = 0$ and a
branch cut of $G_E(\varphi(u))$ along the interval $\cc{\varrho, \cK}$.

\begin{figure}
{\centering
\includegraphics[width=0.48\linewidth]{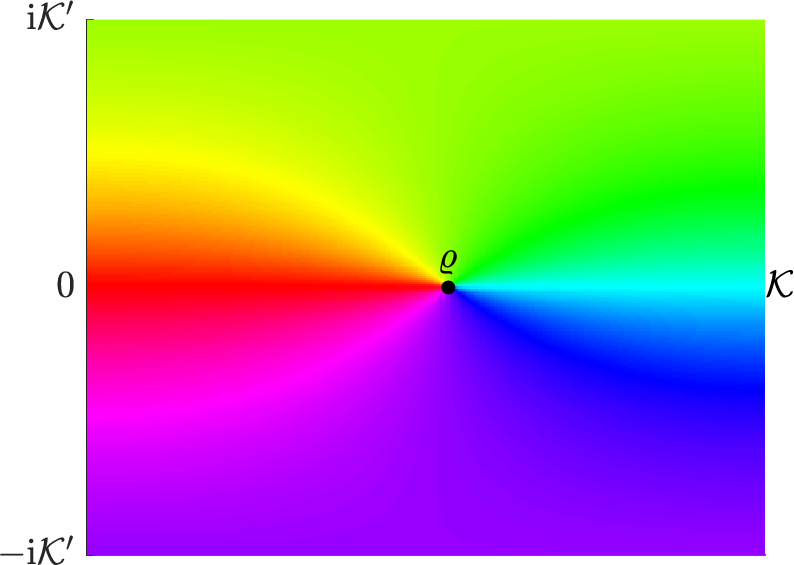}
\includegraphics[width=0.48\linewidth]{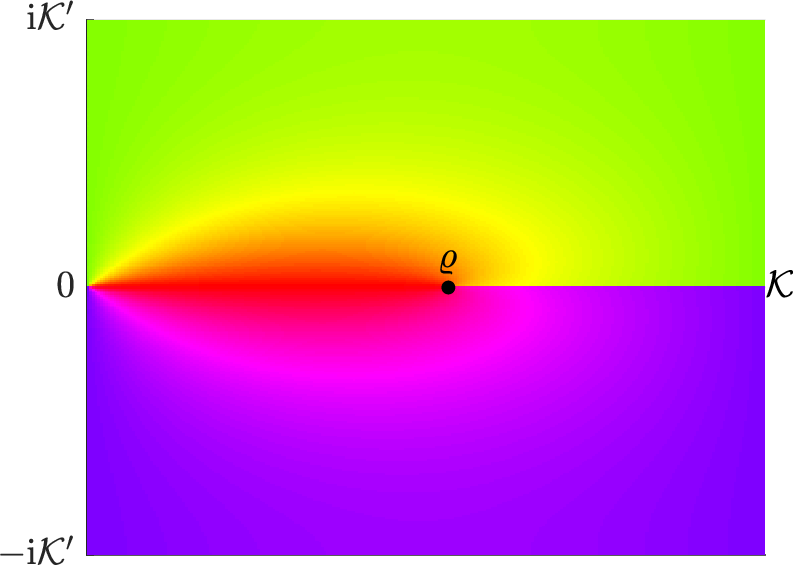}

}
\caption{Phase plot of $\Eta(\varrho+u)/\Eta(\varrho-u)$ (left) and of the 
complex Green's function $G_E(\varphi(u)) = \log \Big( 
\frac{\Eta(\varrho+u)}{\Eta(\varrho-u)} \Big)$ (right) for $u \in P$ and for 
the set $E = \cc{-1, -0.3} \cup \cc{0.1, 1}$.}
\label{fig:GE}
\end{figure}

\item The equilibrium measure $\mu_E$ can also be represented by
\begin{equation} \label{eqn:density_elliptic}
\dd \mu_E(x)
= \frac{2}{\pi \ii} \left( \zn(\varrho) + \frac{\sn(\varrho) \cn(\varrho) \dn(\varrho)}{\sn^2(\varrho) - \sn^2(u)} \right) \dd u
\end{equation}
for $x = \varphi(u) \in E$ and $u \in \cc{0, \ii \cK'} \cup \cc{\cK, \cK + \ii \cK'}$.
The formula can be obtained from~\eqref{eqn:density_muE} or from
$\dd \mu_E(x) = \frac{1}{\pi} \frac{\abs{x - z_*}}{\sqrt{\abs{H(x)}}} \, \dd x$; see~\cite[Thm.~4.2\,(iii)]{SchiefermayrSete2025}.
\end{enumerate}
\end{remark}

Next, we give an explicit formula for the critical value of the Green's function.  The result is essentially contained in~\cite{Schiefermayr2011}, in particular see the proofs of Theorem~3 and Theorem~4 in~\cite{Schiefermayr2011}.

\begin{theorem} \label{thm:critical_value_of_gE}
Let $E = \cc{b_1, b_2} \cup \cc{b_3, b_4}$ be as in~\eqref{eqn:E},
let $k$, $\varrho$ be given as in~\eqref{eqn:k} and~\eqref{eqn:lambda},
and let $z_*$ be the critical point of Green's function $g_E$ given in~\eqref{eqn:z*}, then
\begin{equation} \label{eqn:crit_val_gE}
g_E(z_*) = \log \left( \frac{\Theta(\varrho + v_*)}{\Theta(\varrho - v_*)} \right)
\end{equation}
with the real logarithm and
\begin{equation} \label{eqn:vstar}
v_* = \sn^{-1} \left( \sqrt{\frac{\zn(\varrho)}{k^2 \sn(\varrho) \big( \sn(\varrho) \zn(\varrho) + \cn(\varrho) \dn(\varrho) \big) }} \right).
\end{equation}
\end{theorem}

\begin{proof}
Since $z_* \in \oo{b_2, b_3}$ and by Lemma~\ref{lem:varphi}, there exists a unique
$v_* \in \oo{0, \cK}$ such that $\varphi(v_* + \ii \cK') = z_*$.
By~\eqref{eqn:phi} and~\eqref{eqn:z*}, the latter is equivalent to
\begin{equation*}
\sn^2(v_* + \ii \cK') = \frac{\sn(\varrho) \bigl( \sn(\varrho) \zn(\varrho) + \cn(\varrho) \dn(\varrho) \bigr)}{\zn(\varrho)}.
\end{equation*}
By~\eqref{eqn:snK} and Lemma~\ref{lem:sn-bijective}, we obtain~\eqref{eqn:vstar}.
By~\eqref{eqn:gE} and~\eqref{Eta(uiK)},
\begin{equation*}
g_E(z_*) = \log \abs*{ \frac{\Eta(\varrho + v_* + \ii \cK')}{\Eta(\varrho - v_* - \ii \cK')} }
= \log \left( \frac{\Theta(\varrho + v_*)}{\Theta(\varrho - v_*)} \right),
\end{equation*}
which gives the assertion.
\end{proof}

Now, we are in position to explicitly express all parameters of the 
lemniscatic domain $\comp{L}$ in Theorem~\ref{thm:walsh_map} with the help of 
Jacobi's elliptic and theta functions.

\begin{theorem} \label{thm:lemniscatic_domain_elliptic}
Let $E = \cc{b_1, b_2} \cup \cc{b_3, b_4}$ with $b_1 < b_2 < b_3 < b_4$, 
let $k$ be as in~\eqref{eqn:k},
let $\lambda$ and $\varrho = \lambda \cK$ be as in~\eqref{eqn:lambda},
and let $\comp{L}$ be the lemniscatic domain from Theorem~\ref{thm:walsh_map} 
corresponding to $\comp{E}$, given by
\begin{equation}
L = \{ w \in \C : \abs{w - a_1}^{m_1} \abs{w - a_2}^{m_2} \leq \capacity(E) \}.
\end{equation}
Then, the exponents are
\begin{equation} \label{eqn:m1m2_lambda}
m_1 = 1-\lambda, \quad m_2 = \lambda,
\end{equation}
the logarithmic capacity $\capacity(E)$ is given 
in~\eqref{eqn:capacity_elliptic}, and the centers are
\begin{equation} \label{eqn:aj}
a_1 = \alpha - \lambda \beta, \quad a_2 = \alpha + (1 - \lambda) \beta,
\end{equation}
where
\begin{equation} \label{eqn:alpha_for_2_intervals}
\alpha = \frac{1}{2} (b_1 + b_2 + b_3 + b_4) - z_*
\end{equation}
with $z_*$ in~\eqref{eqn:z*}, and
\begin{equation} \label{eqn:beta_for_2_intervals}
\beta = \frac{(b_4 - b_1)(b_3 - b_1)}{4 (1-\lambda)^{1 - \lambda} \lambda^\lambda (b_2 - b_1)} \frac{\Theta^4(0)}{\Theta^4(\varrho)}
\frac{\Theta(\varrho + v_*)}{\Theta(\varrho - v_*)}
\end{equation}
with $v_*$ in~\eqref{eqn:vstar}.
\end{theorem}

\begin{proof}
By~\eqref{eqn:mj_muE}, $m_1 = \mu_E(\cc{b_1, b_2})$ and $m_2 = 
\mu_E(\cc{b_3, b_4})$.
By Theorem~\ref{thm:GE}\,\ref{it:GE_values_on_Ij}, we have $G_E^-(b_2) = \ii \pi 
\mu_E(\cc{b_1, b_2})$,
and by Theorem~\ref{thm:GE_two_intervals}, $G_E^-(b_2) = 
\ii \pi (1 - \lambda)$.
This shows the first part of~\eqref{eqn:m1m2_lambda}.
The second part follows from $m_2 = 1 - m_1$.
The remaining formulas follow from~\cite[Thm.~3.1 and Thm.~4.5]{SchiefermayrSete2025} by inserting the expressions for 
$\capacity(E)$ from~\eqref{eqn:capacity_elliptic} and $g_E(z_*)$ from Theorem~\ref{thm:critical_value_of_gE}.
\end{proof}

\begin{remark}
If the lengths of the intervals are equal, i.e., if $b_4 - b_3 = b_2 - b_1$, then, since $\sn(\frac{1}{2} \cK) = 1/\sqrt{1+k'}$, the parameter $\lambda$ in~\eqref{eqn:lambda} is $\lambda = \frac{1}{2}$.
Moreover, $\capacity(E) = \frac{1}{2} \sqrt{(b_4-b_3) (b_4-b_2)}$ and $z_* = \frac{1}{2} (b_2+b_3)$, since $\Theta^4(0)/\Theta^4(\frac{1}{2} \cK) = 2 (k')^{3/2} / (1+k')$, see~\cite[Lem.~2]{Schiefermayr2008}, and $\zn(\frac{1}2{} \cK) = \frac{1}{2} (1-k')$ (from~\eqref{eqn:zn_add} with $u = v = \frac{1}{2} \cK$).
This case has also been considered in~\cite[Sect.~3]{SeteLiesen2016}.
\end{remark}

\section{Computation of Walsh's conformal map}\label{sect:conformal_map}

Let us discuss the computation of the conformal map $\Phi$ onto a lemniscatic domain for sets $E$ consisting of two real intervals. An analytic formula for $\Phi$ is known only for two symmetric intervals, $E = \cc{-D, -C} \cup \cc{C, D}$ with $0 < C < D$, see~\cite[Cor.~3.3]{SeteLiesen2016}, and for certain polynomial pre-images, see~\cite[Prop.~3.6, Cor.~3.7]{SchiefermayrSete2023} and~\cite[Thm.~3.3]{SchiefermayrSete2024}. A numerical method for computing $\Phi$ if $E$ consists of several intervals was derived in~\cite[Sect.~5]{SchiefermayrSete2025}. For sets $E$ bounded by $C^2$-smooth Jordan curves, a numerical method for computing $\Phi$ has been derived in~\cite{NasserLiesenSete2016}, and this method can be extended to domains with corners by using a suitable discretization of the boundary, see~\cite{NasserLiesenSete2016}, and, for a more detailed description,~\cite{LiesenSeteNasser2017}. Here, we propose a numerical method similar to the one in~\cite{SchiefermayrSete2025} but based on the real and complex Green's functions and the elliptic characterization.

The relation of the Green's functions of $\comp{E}$ and $\comp{L}$ via the conformal map $\Phi$ leads to an equation whose unique solution is $\Phi(z)$.

\begin{theorem} \label{thm:Phiz}
Let $E = \cc{b_1, b_2} \cup \cc{b_3, b_4}$ with $b_1 < b_2 < b_3 < b_4$.
\begin{enumerate}
\item \label{it:Phi_complex}
For $z \in \C \setminus \R$, $\Phi(z)$ is the unique solution $w \in \C \setminus \R$ of the equation
\begin{equation} \label{eqn:Phi_complex}
G_L(w) = \sum_{j=1}^2 m_j \log(a_j - w) - \log(\capacity(E))
= G_E(z)
\end{equation}
with the principal branch of the logarithm.

\item \label{it:Phi_real}
For $z \in \R \setminus E$, $\Phi(z)$ is the unique solution $w$ of the equation
\begin{equation} \label{eqn:Phi_real}
g_L(w) = \sum_{j=1}^2 m_j \log(\abs{w - a_j}) - \log(\capacity(E))
= g_E(z)
\end{equation}
with $w$ restricted to a suitable interval depending on $z$:
\begin{itemize}
\item $w \in \oc{-\infty, c_1}$ if $z \in \oc{-\infty, b_1}$,
\item $w \in \co{c_2, w_*}$ if $z \in \co{b_2, z_*}$,
\item $w = w_*$ if $z = z_*$,
\item $w \in \oc{w_*, c_3}$ if $z \in \oc{z_*, b_3}$,
\item $w \in \co{c_4, \infty}$ if $z \in \co{b_4, \infty}$.
\end{itemize}
Here, $w_* = m_1 a_2 + m_2 a_1$ is the unique critical point of $g_L$.
\end{enumerate}
\end{theorem}

\begin{proof}
The proof is similar to the proofs of~\cite[Thm.~5.1 and Thm.~5.3]{SchiefermayrSete2025}.

\ref{it:Phi_complex}
Let $z \in \C \setminus \R$.
By Theorem~\ref{thm:GL}\,\ref{it:GL_GE_relation}, $\Phi(z)$ is a solution of~\eqref{eqn:Phi_complex}.
In order to show that the solution is unique, we show that $G_L$ is univalent on $\C \setminus \R$.
Since $\im(G_L(w)) = \sum_{j=1}^2 m_j \arg(a_j - w)$ with the principal branch of the argument, the function $G_L$ satisfies $G_L(\bH^\pm) \subseteq \bH^\mp$ and it is enough to show that $G_L$ is univalent on $\bH^+$ and on $\bH^-$.
The derivative of $G_L$ is $G_L'(w) = \sum_{j=1}^2 m_j/(w-a_j)$, which implies 
$\im(G_L'(w)) = - \sum_{j=1}^2 m_j \im(w)/\abs{w-a_j}^2$ since $a_1, a_2$ are real.
Thus, $\im(G_L'(w)) < 0$ for $w \in \bH^+$, which, by~\cite[Lem.~2.1]{EremenkoYuditskii2012}, shows that $G_L$ is univalent in $\bH^+$.
Similarly, $G_L$ is univalent on $\bH^-$.
Thus, $w = \Phi(z)$ is the unique solution of~\eqref{eqn:Phi_complex}.

\ref{it:Phi_real}
That $\Phi(z)$ is a solution of~\eqref{eqn:Phi_real} follows from $g_L(\Phi(z)) = g_E(z)$, see~\cite{Walsh1958}, \cite[Eqn.~(2.2)]{SeteLiesen2016}, or~\cite[Eqn.~(2.2)]{SchiefermayrSete2023}.
The corresponding intervals in the $z$- and $w$-plane follow from the mapping properties of $\Phi$; see~\cite[Thm.~2.1]{SchiefermayrSete2024}.
The uniqueness of the solution follows from the piecewise strict monotonicity of $g_L : \R \setminus E \to \R$:
\begin{equation*}
g_L'(w) = \sum_{j=1}^2 \frac{m_j}{w - a_j}
\begin{cases}
< 0, & \text{if } w < c_1 < a_1, \\
> 0, & \text{if } c_2 < w < w_*, \\
= 0, & \text{if } w = w_*, \\
< 0, & \text{if } w_* < w < c_3, \\
> 0, & \text{if } w > c_4,
\end{cases}
\end{equation*}
compare also~\cite[Thm.~5.3]{SchiefermayrSete2025}.
\end{proof}

We propose to compute the values $\Phi(z)$ by solving the corresponding equation in Theorem~\ref{thm:Phiz}.  This leads to the following two steps.
\begin{enumerate}
\item The Green's functions are evaluated using the elliptic characterization. For $z \in \R \setminus E$,
\begin{equation}
g_E(z) = \log \abs*{ \frac{\Eta(u + \varrho)}{\Eta(u - \varrho)} }, \quad u = \varphi^{-1}(z),
\end{equation}
with the real logarithm, see Theorem~\ref{thm:Green_function_2_intervals}\,\ref{it:green_function}, and for $z \in \C \setminus \R$,
\begin{equation}
G_E(z) = \log \left( \frac{\Eta(u + \varrho)}{\Eta(u - \varrho)} \right), \quad u = \varphi^{-1}(z),
\end{equation}
with the principal branch of the logarithm; see Theorem~\ref{thm:GE_two_intervals}.
Alternatively, $g_E(z)$ and $G_E(z)$ can be evaluated using their integral 
representation, as was done in~\cite[Sect.~5]{SchiefermayrSete2025}.

\item The value $\Phi(z)$ is obtained by solving the nonlinear equation in  Theorem~\ref{thm:Phiz} with a damped Newton iteration and initial point $w_0$ depending on $z$:
\begin{equation}
w_0 =
\begin{cases}
z, & \text{if } z \in \C \setminus \R, \\
z, & \text{if } z \in \oo{-\infty, b_1} \text{ or } z \in \oo{b_4, \infty}, \\
c_2 + \frac{z - b_2}{z_* - b_2} (w_* - c_2), & \text{if } z \in \oo{b_2, z_*}, \\
w_* + \frac{z-z_*}{b_3 - z_*} (c_3 - w_*), & \text{if } z \in \oo{z_*, b_3}.
\end{cases}
\end{equation}
\end{enumerate}

\begin{figure}
{\centering
\includegraphics[width=0.48\linewidth]{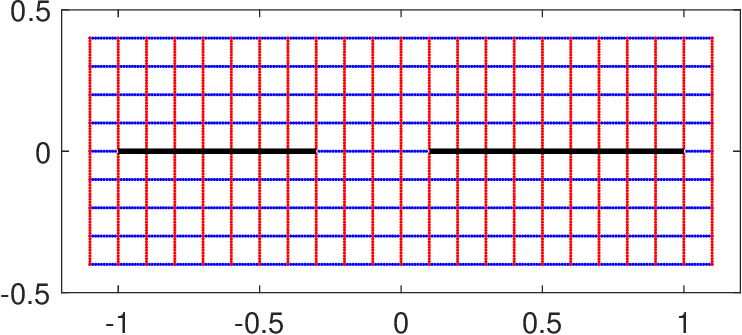}
\includegraphics[width=0.48\linewidth]{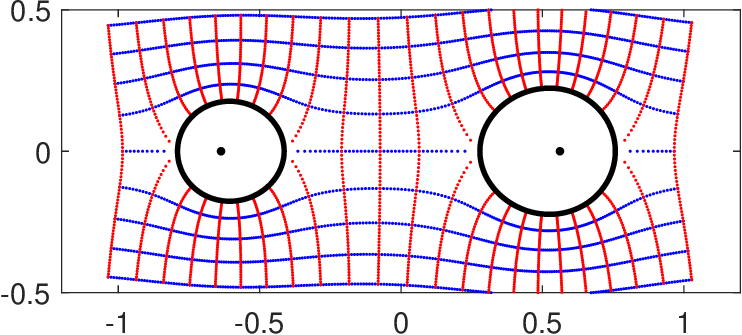}
}
\caption{The conformal map $\Phi$ onto the lemniscatic domain for the 
complement 
of $E = \cc{-1, -0.3} \cup \cc{0.1, 1}$.  Left: Set $E$ (black) and a grid.  
Right: Set $\partial L$ (black) and the image of the grid under $\Phi$.  The 
black dots are the computed centers $a_1, a_2$.}
\label{fig:numerical_map}
\end{figure}

\begin{example} \label{ex:numerical_map}
Consider the union of two intervals $E = \cc{-1, -0.3} \cup \cc{0.1, 1}$ and 
the conformal map $\Phi : \comp{E} \to \comp{L}$ onto a lemniscatic domain.
The parameters of $L$ are computed as indicated in 
Theorem~\ref{thm:lemniscatic_domain_elliptic} and are (rounded to four digits)
\begin{equation*}
m_1 = 0.4671, \quad m_2 = 0.5329, \quad a_1 = -0.6366, \quad a_2 = 0.5619, \quad
\capacity(E) = 0.4898. 
\end{equation*}
We also compute the image of a grid under $\Phi$, where $\Phi(z)$ is evaluated as outlined above; see~Figure~\ref{fig:numerical_map}.

To assess the accuracy of our numerical method, we compare it with two 
methods from the literature.

(i) The method from~\cite{NasserLiesenSete2016} computes Walsh's map and the 
lemniscatic domain for domains with smooth boundary by solving a boundary 
integral equation (BIE) with the Neumann kernel.  Since $E$ does not have a
smooth boundary, we successively apply inverse Joukowsky maps to obtain a 
conformally equivalent domain with a smooth boundary.  Then, the BIE method is 
applied to this smooth domain.
The parameters of the lemniscatic domain $\comp{L}$ computed with the BIE 
method and with our new elliptic method agree to machine precision, more 
precisely,
\begin{align*}
\max \{ \abs{a_1 - a_1^{\text{BIE}}}, \abs{a_2 - a_2^{\text{BIE}}} \} &= 6.7 
\cdot 10^{-16}, \\
\max \{ \abs{m_1 - m_1^{\text{BIE}}}, \abs{m_2 - m_2^{\text{BIE}}} \} &= 5.6 
\cdot 10^{-17}, \\
\abs{\capacity(E) - \capacity(E)^{\text{BIE}}} &= 1.7 \cdot 10^{-16}.
\end{align*}
This suggests that both methods are extremely accurate and yield the exact 
result up to machine precision.

(ii) The method from~\cite{SchiefermayrSete2025} computes Walsh's map and the lemniscatic domain for $\ell$ intervals by solving a nonlinear equation derived from the Green's function, and uses the integral representation of Green's function for the complement of several intervals. The parameters of $\comp{L}$ computed with the method from~\cite{SchiefermayrSete2025} and with the proposed elliptic method agree up to order $10^{-13}$, more precisely, the absolute differences between the centers, exponents and logarithmic capacity are $3.1 \cdot 10^{-13}$, $5.1 \cdot 10^{-14}$, and $2.3 \cdot 10^{-13}$, respectively. Again, this suggests that both methods are very accurate. In the method~(ii), the Green's function is given as an integral, where the integrand has singularities at each endpoint of the intervals of integration. We suspect that the numerical evaluation of these integrals is exact up to order $10^{-13}$ in our implementation, which leads to the differences of order $10^{-13}$ compared to the new method.
\end{example}

\section{Jacobi's elliptic functions}\label{sect:elliptic_functions}

In this section, we collect some properties of Jacobi's theta and elliptic 
functions.  Additionally, we prove a result about the function 
$\Eta(v+u)/\Eta(v-u)$ in Theorem~\ref{thm:Eta_quotient}, that is needed in our 
investigation of the complex Green's function in 
Theorem~\ref{thm:GE_two_intervals}.

Let $k \in \oo{0, 1}$ be the modulus,
$k' \coloneq \sqrt{1-k^2} \in \oo{0, 1}$ be the complementary modulus,
$\cK \equiv \cK(k)$ be the complete elliptic integral of the first kind,
and $\cK' \equiv \cK'(k) \coloneq \cK(k')$ be the associated complete 
elliptic integral of the first kind.  Note that $\cK, \cK' > 0$.

Then, Jacobi's theta functions $\Theta(u) \equiv \Theta(u, k)$, $\Eta(u) \equiv \Eta(u, k)$, $\Eta_1(u) \equiv \Eta_1(u, k)$, $\Theta_1(u) \equiv \Theta_1(u, k)$ can be expressed with the help of the Fourier series
\begin{equation}
\begin{aligned}
\Theta(u) &\equiv \theta_4(v)
\coloneq 1 + 2 \sum_{n=1}^\infty (-1)^n q^{n^2} \cos(2nv), \\
\Eta(u) &\equiv \theta_1(v)
\coloneq 2 \sum_{n=0}^\infty (-1)^n q^{(n+1/2)^2} \sin((2n+1)v), \\
\Eta_1(u) &\equiv \theta_2(v)
\coloneq 2 \sum_{n=0}^\infty q^{(n+1/2)^2} \cos((2n+1)v), \\
\Theta_1(u) &\equiv \theta_3(v)
\coloneq 1 + 2 \sum_{n=1}^\infty q^{n^2} \cos(2nv),
\end{aligned}
\end{equation}
where $v = \pi u/(2\cK)$ and $q=\exp(-\pi\cK'/\cK)$,
see~\cite[1050.01]{ByrdFriedman1971} or \cite[p.~5]{Lawden1989}.
Here $\Eta$ denotes the Greek capital eta.
In view of~\eqref{eqn:sncndn-Theta}, we prefer to use Jacobi's old notation $\Theta(u), \Eta(u), \Eta_1(u), \Theta_1(u)$ instead of the more recent notation $\theta_1(v), \theta_2(v), \theta_3(v)$, $\theta_4(v)$.

With the help of these functions, one can define Jacobi's elliptic functions $\sn(u)\equiv\sn(u,k)$, $\cn(u)\equiv\cn(u,k)$, $\dn(u)\equiv\dn(u,k)$ by
\begin{equation}\label{eqn:sncndn-Theta}
\sn(u) \coloneq \frac{1}{\sqrt{k}} \frac{\Eta(u)}{\Theta(u)}, \quad
\cn(u) \coloneq \frac{\sqrt{k'}}{\sqrt{k}} \frac{\Eta_1(u)}{\Theta(u)}, \quad
\dn(u) \coloneq \sqrt{k'} \frac{\Theta_1(u)}{\Theta(u)},
\end{equation}
see~\cite[1052.02]{ByrdFriedman1971}.
All these functions are seen as functions of the variable $u$ but also depend 
on the modulus $k$.
For simplicity of notation, we drop the argument $k$ when there is no 
possibility of misunderstanding and write $\Eta(u), \sn(u), \cn(u)$~etc.\ 
instead of $\Eta(u,k), \sn(u,k), \cn(u,k)$ etc.
Jacobi's elliptic functions satisfy the well-known 
identities~\cite[121.00]{ByrdFriedman1971}
\begin{equation} \label{eqn:sncndn}
\sn^2(u)+\cn^2(u)=1, \quad
k^2\sn^2(u)+\dn^2(u)=1, \quad
\dn^2(u) - k^2 \cn^2(u) = {k'}^2
\end{equation}
and~\cite[122.03, 122.07]{ByrdFriedman1971}
\begin{equation} 
\sn(u+\cK) = \frac{\cn(u)}{\dn(u)}, \quad
\sn(u+\iK) = \frac{1}{k \sn(u)},
\label{eqn:snK}
\end{equation}
and~\cite[125.02, 120.02]{ByrdFriedman1971}
\begin{equation}
\sn(\ii u,k) = \ii \frac{\sn(u, k')}{\cn(u, k')}, \quad
\cn(\ii u, k) = \frac{1}{\cn(u, k')}, \quad
\dn(\ii u, k) = \frac{\dn(u, k')}{\cn(u, k')}. \label{eqn:sncndniu}
\end{equation}


The next lemma follows immediately from the fundamental properties of the conformal map $\sn$.


\begin{lemma}[{\cite[Sect.~6.24]{AndersonVamanamurthyVuorinen1997}}] \label{lem:sn-bijective}
For fixed $k \in \oo{0, 1}$, the function $\sn : \cc{0, \cK} \to \cc{0, 1}$
is bijective and we denote its inverse by $\sn^{-1} : \cc{0, 1} \to \cc{0, \cK}$.
\end{lemma}

In the following lemma, we collect some well-known properties of $\Eta(u)$ and $\Theta(u)$, see~\cite[Ch.~1]{Lawden1989}.


\begin{lemma}\label{Lemma-ThetaEta}
Let $k \in \oo{0,1}$, $\cK \equiv \cK(k)$, and $\cK' \equiv \cK(k')$.
\begin{enumerate}
\item Jacobi's theta function $\Eta(u)$ is an odd analytic function in $\C$ with period $4\cK$ and has its zeros at $u = 2 m \cK + 2 \ii m' \cK'$ with $m, m' \in \Z$, where all zeros are simple.  Furthermore, $\Eta(u) > 0$ for $u \in \oo{0, 2 \cK}$ and the following identities hold:
\begin{align}
\Eta(\conj{u}) &= \conj{\Eta(u)}, \label{Eta(ovu)} \\
\Eta(u \pm \cK) &= \pm \Eta_1(u), \label{Eta(uK)} \\
\Eta(u \pm 2 \cK) &= - \Eta(u), \label{Eta(u2K)} \\
\Eta(u \pm \iK) &= \pm \ii \ee^{\pi \cK'/(4\cK)} \ee^{\mp \ii \pi u/(2\cK)} \Theta(u), \label{Eta(uiK)} \\
\Eta(2 u) &= \frac{2 \sqrt{k} \Theta^4(u) \sn(u) \cn(u) \dn(u)}{\Theta^3(0)},\label{eqn:Eta(2u)} \\
\sn^2(u) - \sn^2(v) &= \frac{\Eta(u-v) \Eta(u+v) \Theta^2(0)}{k \Theta^2(u) \Theta^2(v)}.\label{eqn:snsn-Eta}
\end{align}

\item \label{it:Theta}
Jacobi's theta function $\Theta(u)$ is an even analytic function in $\C$ with period $2\cK$ and has its zeros at $u = \iK + 2 m \cK + 2 \ii m' \cK'$ with $m, m' \in \Z$, where all zeros are simple.  Furthermore, $0 < \Theta(0) \leq \Theta(u) \leq \Theta(\cK)$ for $u \in \R$.
\end{enumerate}
\end{lemma}

Next, we introduce Jacobi's zeta function and collect some of its properties; see~\cite[Sect.~3.6]{Lawden1989}.

\begin{lemma} \label{lem:zeta}
Let $k \in \oo{0,1}$.
Jacobi's zeta function $\zn(u) \equiv \zn(u,k)$, defined by
\begin{equation}\label{zn}
\zn(u) \coloneq \frac{\Theta'(u)}{\Theta(u)},
\end{equation}
is an odd function with period $2\cK$, has zeros at $u = m \cK$ with $m \in \Z$, and satisfies $\zn(u) > 0$ for $u \in \oo{0, \cK}$.
It satisfies the identities
\begin{align}
\zn(u \pm v) &= \zn(u) \pm \zn(v) \mp k^2 \sn(u) \sn(v) \sn(u \pm v), \label{eqn:zn_add} \\
\zn(u \pm \cK) &= \zn(u) - \frac{k^2 \sn(u) \cn(u)}{\dn(u)}, \label{eqn:znK} \\
\zn(u+v) - \zn(u-v) 
&= 2 \zn(v) - 2 k^2 \sn^2(u) \frac{\sn(v) \cn(v) \dn(v)}{1 - k^2 \sn^2(u) 
\sn^2(v)}. \label{eqn:zn_add_difference}
\end{align}
\end{lemma}

In Lemma~\ref{lem:zeta}, we follow the notation of Carlson and 
Todd~\cite{CarlsonTodd1983_elliptic} for Jacobi's zeta function, in other 
references, like~\cite[Sect.~3.6]{Lawden1989} and~\cite[p.~33]{ByrdFriedman1971}, Jacobi's zeta function is 
denoted by a Greek capital zeta $\text{Z}$.


In the next theorem, we consider an auxiliary function appearing in the 
elliptic representation of the complex Green's function in 
Theorem~\ref{thm:GE_two_intervals}.

\begin{theorem} \label{thm:Eta_quotient}
Fix $k \in \oo{0, 1}$, $v \in \oo{0, \cK}$, and consider the function
\begin{equation}
f(u) \coloneq \frac{\Eta(v+u)}{\Eta(v-u)}.
\end{equation}
\begin{enumerate}
\item \label{it:dlogf}
The logarithmic derivative of $f$ is
\begin{equation} \label{eqn:dlogf}
\frac{f'(u)}{f(u)}
= \frac{\dd}{\dd u} \log \left( \frac{\Eta(v + u)}{\Eta(v - u)} \right)
= 2 \zn(v) + \frac{2 \sn(v) \cn(v) \dn(v)}{\sn^2(v) - \sn^2(u)}.
\end{equation}

\item \label{it:0_iKp_K_KiKp}
The function $f$ satisfies $\abs{f(u)} = 1$ for $u \in \cc{0, \ii \cK'} 
\cup \cc{\cK, \cK + \ii \cK'}$.

\item \label{it:f_on_boundary}
The function $f$ has the following mapping properties.
\begin{itemize}
\item $f : \co{0, v} \to \co{1, \infty}$ and $f : \oc{v, \cK} \to \oc{-\infty, 
-1}$ are strictly increasing and bijective.

\item $f$ maps $\cc{\cK, \cK + \ii \cK'}$ bijectively onto the arc 
of the upper unit circle from $f(\cK) = -1$ to $f(\cK + \ii \cK') = \ee^{\ii 
\pi (1 - v/\cK)}$.

\item $f$ maps $\cc{0, \ii \cK'}$ bijectively onto the arc of the upper unit 
circle from $f(0) = 1$ to $f(\ii \cK') = \ee^{\ii \pi (1 - v/\cK)}$.

\item There exists $r > 1$ such that $f :  \cc{\ii \cK', \cK + \ii \cK'} \to 
\cc{\ee^{\ii \pi (1 - v/\cK)}, r \ee^{\ii \pi (1 - v/\cK)}}$ is $2$-to-$1$.
\end{itemize}

\item \label{it:image_of_rectangle}
With $r > 1$ from~\ref{it:f_on_boundary} and $P^\pm$ from~\eqref{eqn:varphi_Pplus}, the functions
\begin{align}
&f : P^+
\to \bH^+ \setminus 
(\overline{\bD} \cup \cc{ \ee^{+ \ii \pi (1-v/\cK)}, r \ee^{+ \ii \pi 
(1-v/\cK)}}), \label{eqn:image_P+} \\
&f : P^-
\to \bH^- \setminus 
(\overline{\bD} \cup \cc{\ee^{- \ii \pi (1 - v/\cK)}, r \ee^{- \ii \pi (1 - 
v/\cK)}} ), \label{eqn:image_P-} \\
&f : -P^- 
\to (\bH^+ \cap \bD) 
\setminus \cc{r^{-1} \ee^{+ \ii \pi (1 - v/\cK)}, \ee^{+ \ii \pi (1 - v/\cK)}},
\\
&f : 
- P^+
\to (\bH^- \cap \bD) 
\setminus \cc{r^{-1} \ee^{- \ii \pi (1 - v/\cK)}, \ee^{- \ii \pi (1 - v/\cK)}}
\label{eqn:image_-P+}
\end{align}
are conformal and bijective, where $\bD = \{ z \in \C : \abs{z} < 1 \}$.

\item \label{it:f_on_domain}
For $u \in P^+ \cup P^- \cup \oo{0, v}$, the function $\log(f(u))$ with the 
principal branch of the logarithm is defined, is real on $\oo{0, v}$ and 
satisfies $\log(f(0)) = 0$.
\end{enumerate}
\end{theorem}

Theorem~\ref{thm:Eta_quotient}\,\ref{it:f_on_boundary} 
and~\ref{it:image_of_rectangle} are illustrated in 
Figure~\ref{fig:Eta_quotient}.

\begin{figure}
{\centering
\begin{tikzpicture}[scale=1.1,
declare function = {
r = 1.4;
myangle = 65;
v = (1 - myangle/180)*1.5;  
}]


\fill[blue!40] (0,0) rectangle (1.5,2);
\fill[blue!20] (0,0) rectangle (1.5,-2);
\fill[blue!10] (-1.5,0) rectangle (0,2);
\fill[blue!30] (-1.5,0) rectangle (0,-2);

\draw[->] (-2,0) -- (2,0) node[right] {$\re$};
\draw[->] (0,-2.4) -- (0,2.4) node[above] {$\im$};

\draw[thick] (-1.5,-2) rectangle (1.5,2);

\draw[thick] (v,+0.06) -- (v,-0.06) node[below] {$v$};

\draw[ultra thick, mid arrow, green!60!black] (v,0) -- (1.5,0);
\draw[ultra thick, mid arrow, teal!60!black] (1.5,0) -- (1.5,2);
\draw[ultra thick, mid arrow, cyan!60!black] (1.5,2) --(0,2);
\draw[ultra thick, mid arrow, blue!60!black] (0,2) -- (0,0);
\draw[ultra thick, mid arrow, violet!80!black] (0,0) -- (v,0);

\node[below left] at (0,0) {0};
\node[below right] at (1.5,0) {$\cK$};
\node[below left] at (-1.5,0) {$-\cK$};
\node[above left] at (0,2) {$\ii \cK'$};
\node[below left] at (0,-2) {$-\ii \cK'$};
\node[] at (0.75, 1) {$P^+$};
\node[] at (0.75, -1) {$P^-$};

\draw[->, thick] (2.4,0.5) -- (3.2,0.5) node[midway, above] {$f$};


\begin{scope}[shift={(6.5,0)}]

\fill[blue!40] (-3,0) rectangle (3,2.4);
\fill[blue!20] (-3,0) rectangle (3,-2.4);
\begin{scope}
\clip (0,0) circle (1);
\fill[blue!10] (-1,0) rectangle (1,1);
\end{scope}
\begin{scope}
\clip (0,0) circle (1);
\fill[blue!30] (-1,-1) rectangle (1, 0);
\end{scope}

\draw[->] (-3,0) -- (3,0) node[right] {$\re$};
\draw[->] (0,-2.4) -- (0,2.4) node[above] {$\im$};

\node[below left] at (-1,0) {$-1$};
\node[below right] at (1,0) {$1$};
\node[below right] at (0,0) {0};

\draw[thick] (0,0) circle (1);

\draw[ultra thick, black] ({1/r * cos(myangle)},{1/r * sin(myangle)}) -- 
({r*cos(myangle)},{r*sin(myangle)}) node[above right] {$r \ee^{\ii \pi(1 - 
v/\cK)}$};

\draw[ultra thick, black] ({1/r * cos(myangle)},{1/r * sin(-myangle)}) -- 
({r*cos(-myangle)},{r*sin(-myangle)});

\draw[ultra thick, mid arrow, green!60!black] (-3,0) -- (-1,0);
\draw[ultra thick, mid arrow, teal!60!black] (-1,0) arc[start angle=180, end 
angle=65, radius=1cm];
\draw[ultra thick, cyan!60!black] ({cos(myangle)},{sin(myangle)}) -- 
({r*cos(myangle)},{r*sin(myangle)}) node[above right] {\color{black} $r 
\ee^{\ii \pi(1 - 
v/\cK)}$};
\draw[ultra thick, mid arrow reversed, blue!60!black] (1,0) arc[start angle=0, 
end angle=65, radius=1cm];
\draw[ultra thick, mid arrow, violet!80!black] (1,0) -- (3,0);
\end{scope}

\end{tikzpicture}

}
\caption{Illustration of the map $f$ in Theorem~\ref{thm:Eta_quotient}.}
\label{fig:Eta_quotient}
\end{figure}
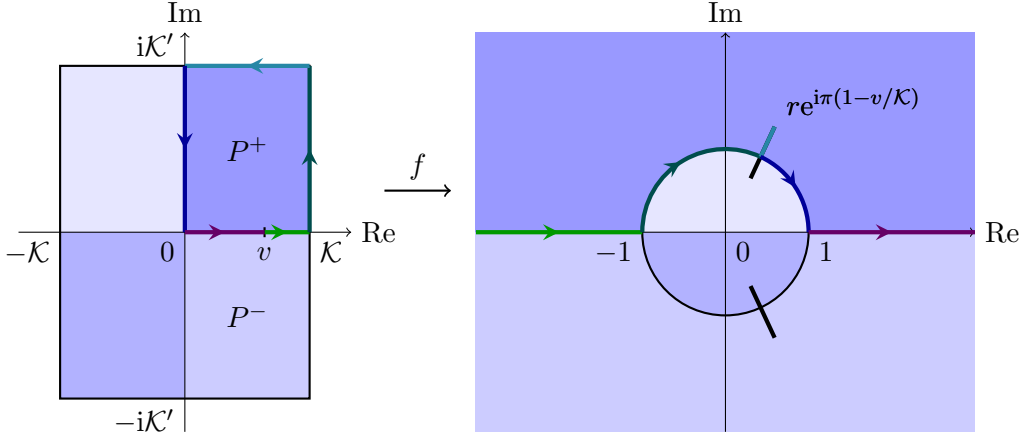

\begin{proof}
\ref{it:dlogf} has been obtained in~\cite[Lem.~7.3]{Schiefermayr2015_Zolotarev}.

\ref{it:0_iKp_K_KiKp}
We write $u \in \cc{0, \ii \cK'}$ as $u = \ii t$, $t \in \cc{0, \cK'}$.  Then, 
by~\eqref{Eta(ovu)},
\begin{equation*}
f(u) = \frac{\Eta(v + \ii t)}{\Eta(v - \ii t)}
= \Eta(v + \ii t) \Big/ \conj{\Eta(v + \ii t)},
\end{equation*}
showing $\abs{f(u)}  = 1$.  Similarly, $u \in \cc{\cK, \cK + \ii \cK'}$ has the 
form $u = \cK + \ii t$ with $t \in \cc{0, \cK'}$.  By~\eqref{Eta(u2K)} 
and~\eqref{Eta(ovu)},
\begin{equation*}
f(u) = \frac{\Eta(v + \cK + \ii t)}{\Eta(v - \cK - \ii t)}
= - \frac{\Eta(v + \cK + \ii t)}{\Eta(v + \cK - \ii t)}
= - \Eta(v + \cK + \ii t) \Big/ \conj{\Eta(v + \cK + \ii t)},
\end{equation*}
hence $\abs{f(u)} = 1$.

\ref{it:f_on_boundary}
We consider the four intervals $\cc{0, \cK}$, $\cc{0, \ii \cK'}$, $\cc{\cK, \cK + 
\ii \cK'}$ and $\cc{\ii \cK', \cK + \ii \cK'}$ separately.

\emph{Interval $\cc{0, \cK}$.} If $u \in \oo{0, \cK}$ then $v + u \in \oo{0, 2 \cK}$ and 
$\Eta(v+u) > 0$.  At $u = v$, $f$ has a simple pole.
For $u \in \oo{0, v}$, $v-u \in \oo{0, \cK}$, hence $\Eta(v-u) > 0$ and 
$f(u) > 0$.
For $u \in \oo{v, \cK}$, $v-u \in \oo{-\cK, 0}$, hence $\Eta(v-u) < 0$ and 
$f(u) < 0$.

Next, we show that $f'(u) > 0$ on $\oo{0, v} \cup \oo{v, \cK}$.
By~\ref{it:dlogf},
\begin{equation*}
f'(u) = 2 f(u) \left( \zn(v) + \frac{\sn(v) \cn(v) \dn(v)}{\sn(v)^2 - \sn(u)^2} 
\right).
\end{equation*}
We investigate the sign of the term in parentheses.

For $u \in \oo{0, v}$, we have that $0 < \sn(u) < \sn(v)$, hence the term in 
parentheses is positive, and $f'(u) > 0$.
Since $f(0) = 1$ and $\lim_{u \to v, u < v} f(u) = + \infty$, this shows that $f 
: \co{0, v} \to \co{1, \infty}$ is strictly increasing and bijective.

For $u \in \oo{v, \cK}$, we have that $\sn(v) < \sn(u) < 1$.
Then,
\begin{equation*}
\zn(v) + \frac{\sn(v) \cn(v) \dn(v)}{\sn(v)^2 - \sn(u)^2}
< \zn(v) + \frac{\sn(v) \cn(v) \dn(v)}{\sn(v)^2 - 1}
= \zn(v) - \frac{\sn(v) \dn(v)}{\cn(v)}.
\end{equation*}
Using~\eqref{eqn:sncndn} and Lemma~\ref{lem:zeta}, we get
\begin{align*}
&\zn(v) - \frac{\sn(v) \dn(v)}{\cn(v)}
= \zn(v-\cK) + k^2 \frac{\sn(v) \cn(v)}{\dn(v)} - \frac{\sn(v) \dn(v)}{\cn(v)} 
\\
&= \zn(v-\cK) + \sn(v) \frac{k^2 \cn(v)^2 - \dn(v)^2}{\cn(v) \dn(v)}
= - \zn(\cK-v) - \frac{{k'}^2 \sn(v)}{\cn(v) \dn(v)} < 0.
\end{align*}
Thus, $f'(u) > 0$.
Since $\lim_{u \searrow v} f(u) = - \infty$ and
$f(\cK) = \frac{\Eta(v+\cK)}{\Eta(v-\cK)} = - \frac{\Eta(v+\cK)}{\Eta(v+\cK)} = 
-1$, this shows that $f : \oc{v, \cK} \to \oc{-\infty, -1}$ is strictly increasing and bijective.

\emph{Interval $\cc{0, \ii \cK'}$.}  Let $u \in \cc{0, \ii \cK'}$.  By~\ref{it:0_iKp_K_KiKp}, 
$\abs{f(u)} = 1$, which implies $\log(f(u)) = \ii \arg(f(u))$.
By~\eqref{eqn:dlogf}, for $t \in \oo{0, \cK'}$,
\begin{equation*}
\frac{\dd}{\dd t} \arg(f(\ii t))
= \frac{1}{\ii} \frac{\dd}{\dd t} \log(f(\ii t)) 
= 2 \zn(v) + \frac{2 \sn(v) \cn(v) \dn(v)}{\sn^2(v) - \sn^2(\ii t)}.
\end{equation*}
By~\eqref{eqn:sncndniu},
$\sn^2(\ii t, k) = - \sn^2(t, k') / \cn^2(t, k') < 0$.
Moreover, $\sn(v), \cn(v), \dn(v)$, $\zn(v) > 0$ for $v \in \oo{0, \cK}$, hence
$\frac{\dd}{\dd t} \arg(f(\ii t)) > 0$, $t \in \oo{0, \cK'}$.
Therefore, $t \mapsto \arg(f(\ii t))$ is strictly increasing on $\cc{0, \cK'}$.
Finally, $f(0) = 1$ since $\Eta(v) \neq 0$ for $v \in \oo{0, \cK}$, and, 
by~\eqref{Eta(uiK)}, $f(\ii \cK') = \ee^{\ii \pi (1 - v/\cK)}$.

\emph{Interval $\cc{\cK, \cK + \ii \cK'}$.}
Let $u \in \cc{\cK, \cK + \ii \cK'}$.  By~\ref{it:0_iKp_K_KiKp}, 
$\abs{f(u)} = 1$, which implies $\log(f(u)) = \ii \arg(f(u))$.
Next, we show that $\arg(f) : \cc{\cK, \cK + \ii \cK'} \to \R$ is strictly 
decreasing.  Write $u = \cK + \ii t$ with $t \in \cc{0, \cK'}$.
By~\eqref{eqn:dlogf},  we get
\begin{equation*}
\frac{\dd}{\dd t} \arg(f(\cK + \ii t))
= \frac{1}{\ii} \frac{\dd}{\dd t} \log(f(\cK + \ii t)) 
= 2 \zn(v) + \frac{2 \sn(v) \cn(v) \dn(v)}{\sn^2(v) - \sn^2(\cK + \ii t)}.
\end{equation*}
By~\eqref{eqn:snK} and~\eqref{eqn:sncndniu}, we have
$\sn(\cK + \ii t, k) = 1/\dn(t, k')$.
Since $k \leq \dn(t, k') \leq 1$, we obtain that
$\sn^2(v) - k^{-2} \leq \sn^2(v) - \sn^2(\cK + \ii t) \leq \sn^2(v) - 1 < 0$, 
and thus, by~\eqref{eqn:sncndn}, \eqref{eqn:znK} and Lemma~\ref{lem:zeta},
\begin{align*}
\frac{\dd}{\dd t} \arg(f(\cK + \ii t))
&\leq 2 \left( \zn(v) + \frac{\sn(v) \cn(v) \dn(v)}{\sn^2(v) - k^{-2}} \right) 
= 2 \left( \zn(v) + k^2 \frac{\sn(v) \cn(v) \dn(v)}{k^2 \sn^2(v) - 1} \right) 
\\
&= 2 \left( \zn(v) - k^2 \frac{\sn(v) \cn(v) }{\dn(v)} \right)
= 2 \zn(v + \cK) < 0.
\end{align*}
Finally, $f(\cK) = \ee^{\ii \pi}$ is established with~\eqref{Eta(uK)}, and
$f(\cK + \ii \cK') = \ee^{\ii \pi ( 1 - v/\cK)}$ with~\eqref{Eta(uiK)}
and the $2 \cK$-periodicity of $\Theta$, see Lemma~\ref{Lemma-ThetaEta}\,(ii).

\emph{Interval $\cc{\ii \cK', \cK + \ii \cK'}$.}
Let $u \in \cc{\ii \cK', \cK + \ii \cK'}$, then $u = t + \ii \cK'$ with $t \in 
\cc{0, \cK}$.  By~\eqref{Eta(uiK)},
\begin{equation*}
f(t + \ii \cK')
= \frac{\Eta(v + t + \ii \cK')}{\Eta(v - t - \ii \cK')}
= - \ee^{- \ii \pi v/\cK} \frac{\Theta(v + t)}{\Theta(v - t)}
= \ee^{\ii \pi (1 - v/\cK)} \frac{\Theta(v + t)}{\Theta(v - t)}.
\end{equation*}
Since the theta function has real values, see Lemma~\ref{Lemma-ThetaEta}\,\ref{it:Theta}, $f(t + \ii \cK')$ has fixed argument 
and variable modulus.
We investigate the quotient of theta functions.  Let
\begin{equation*}
g : \cc{0, \cK} \to \oo{0, \infty}, \quad
g(t) \coloneq \frac{\Theta(v + t)}{\Theta(v - t)}.
\end{equation*}
Then, $g(0) = 1$, and $g(\cK) = 1$, since $\Theta$ is $2 \cK$-periodic.
For $t \in \cc{0, \cK}$,
\begin{equation*}
g'(t)
= \frac{\Theta(t+v)}{\Theta(t-v)} \left( \frac{\Theta'(t+v)}{\Theta(t+v)} - 
\frac{\Theta'(t-v)}{\Theta(t-v)} \right)
= g(t) (\zn(t+v) - \zn(t-v)).
\end{equation*}
The first factor $g(t)$ is positive.
By~\eqref{eqn:zn_add_difference},
\begin{align*}
g'(t) = 0
&\Leftrightarrow
\zn(t+v) - \zn(t-v) = 0 \\
& \Leftrightarrow
\sn(t) = + \sqrt{\frac{\zn(v)}{k^2 \sn(v) (\cn(v) \dn(v) + \sn(v) \zn(v))}},
\end{align*}
since $\sn(t) \geq 0$ for $t \in \cc{0, \cK}$.  Additionally, $\sn$ is strictly 
increasing on $\cc{0, \cK}$, hence $g$ has a unique critical point $t_* \in 
\oo{0, \cK}$.
Together with $g'(0) = 2 \zn(v) > 0$ and,
since $\zn$ is odd and $2 \cK$-periodic,
\begin{equation*}
g'(\cK) = \zn(\cK + v) - \zn(\cK - v) = 2 \zn(v+\cK) < 0,
\end{equation*}
we find that
$g$ is strictly increasing on $\cc{0, t_*}$ and strictly decreasing on 
$\cc{t_*, \cK}$.
Then, with $r \coloneq g(t_*) > g(0) = 1$, we find that
$f :  \cc{\ii \cK', \cK + \ii \cK'} \to 
\cc{\ee^{\ii \pi (1 - v/\cK)}, r \ee^{\ii \pi (1 - v/\cK)}}$ is $2$-to-$1$.

\ref{it:image_of_rectangle}
By~\ref{it:f_on_boundary}, we have that $f : \partial P^+ \to \partial ( 
\bH^+ \setminus ( \overline{\bD} \cup \cc{\ee^{\ii \pi (1 - v/\cK)}, r 
\ee^{\ii \pi (1 - v/\cK)} } )$.  By distinguishing the two sides of the 
spike, this map is bijective on the boundary and hence conformal and bijective 
from $P^+$ to $\bH^+ \setminus (\overline{\bD} \cup \cc{\ee^{\ii \pi (1 - 
v/\cK)}, r \ee^{\ii \pi (1 - v/\cK)} })$.
The assertions for the maps~\eqref{eqn:image_P-}--\eqref{eqn:image_-P+} follow 
from~\eqref{eqn:image_P+} and the Schwarz reflection principle, or, in a more 
elementary way, with $f(u) = \conj{f(\conj{u})}$ and $f(-u) = 1/f(u)$.

\ref{it:f_on_domain}
That $f : P^+ \cup P^- \cup \oo{0, v} \to \C \setminus \oc{-\infty, 0}$ 
follows from~\ref{it:image_of_rectangle} and~\ref{it:f_on_boundary}.
Thus, the principal branch of the logarithm is well-defined on the image of $f$.
Since $f(\co{0,v}) = \co{1, \infty}$, $\log(f(u))$ is real for $u \in \oo{0,v}$ 
with $\log(f(0)) = \log(1) = 0$.
\end{proof}

\appendix
\section{Appendix}

\begin{proof}[Proof of Lemma~\ref{lem:varphi}]
The properties of $\varphi$ follow from the representation
$\varphi(u) = T(\sn^2(u))$ with the M\"obius transformation
\begin{equation*}
T(w) \coloneq \frac{b_2 w - b_1 \sn^2(\varrho)}{w - \sn^2(\varrho)},
\end{equation*}
and from the known properties of~$\sn$.
The function $\sn$ is an odd elliptic function of order~$2$ with fundamental periods $\{ 4 \cK$, $2 \ii \cK' \}$; see~\cite[Sect.~6.24]{AndersonVamanamurthyVuorinen1997}. Together with $\sn(u + 2 \cK) = - \sn(u)$, we see that $\sn^2$ is an even elliptic function of order~$2$ with fundamental periods $\{ 2 \cK, 2 \ii \cK' \}$ (note that the real period is halved).  Hence, also $\varphi(u) = T(\sn^2(u))$ is an even elliptic function of order~$2$ with fundamental periods $\{ 2 \cK, 2 \ii \cK' \}$.
In particular, $\varphi$ has exactly two poles (counting multiplicity) in a 
period parallelogram.  Since $\pm \varrho$ are distinct poles of $\varphi$, 
these are simple.

Next, we give a constructive proof of the mapping properties of $\varphi$, illustrated in Figure~\ref{fig:varphi}.
By~\cite[Sect.~6.24]{AndersonVamanamurthyVuorinen1997}, 
$\sn : P^+ \to \{ z \in \C : \re(z) > 0, \im(z) > 0 \}$ is conformal and bijective with the boundary correspondences
\begin{equation*}
\begin{aligned}
&\sn : \co{0, \cK} \to \co{0, 1}, && \sn : \oc{\cK, \cK + \ii \cK'} \to \oc{1, 1/k}, \\
&\sn : \oo{\cK + \ii \cK', \ii \cK'} \to \oo{1/k, \infty}, &&
\sn : \oo{0, \ii \cK'} \to \{ \ii y : y \in \oo{0, \infty} \}.
\end{aligned}
\end{equation*}
Since $\sn(u) = \conj{\sn(\conj{u})}$, $u \in \C$, we get that
$\sn : R \coloneq P^+ \cup P^- \cup \oo{0, \cK}
\to \{ z \in \C : \re(z) > 0 \} \setminus \co{1, \infty}$ and thus
$\sn^2 : R \to \C \setminus ( \oc{-\infty, 0} \cup \co{1, \infty} )$ are conformal and bijective.
Since
\begin{equation*}
T(0) = b_1, \quad
T(\sn^2(\varrho)) = \infty, \quad
T(1) = b_4, \quad
T(1/k^2) = b_3, \quad
T(\infty) = b_2,
\end{equation*}
the M\"obius transformation $T$ maps the real line onto the real line but reverses the orientation.  In particular, $\varphi : P^\pm \to \bH^\mp$, and the boundary correspondences~\eqref{eqn:phi_boundary_correspondence} hold.
Finally, also~\eqref{eqn:phi_onto_E_complement} is conformal and bijective.
The formula for $\varphi^{-1}$ follows with $T^{-1}(z) = \sn^2(\varrho) \frac{z-b_1}{z-b_2}$ and by reversing the three steps of the construction.
\end{proof}

\begin{proof}[Proof of Theorem~\ref{thm:Green_function_2_intervals}]
\ref{it:green_function}, \ref{it:logarithmic_capacity}
We show that $g_E$ defined in~\eqref{eqn:gE} is indeed Green's function of $\comp{E}$.
First, $g_E$ is harmonic in $\C \setminus E$ as real part of an analytic function.
Next, we show that $g_E$ has boundary values zero.
The points $z = \varphi(u) \in E = \cc{b_1, b_2} \cup \cc{b_3, b_4}$ correspond to $u \in \cc{0, \ii \cK'} \cup \cc{\cK, \cK + \ii \cK'}$.
Thus, by Theorem~\ref{thm:Eta_quotient}\,\ref{it:0_iKp_K_KiKp},
$g_E(z) = \log \abs{ \frac{\Eta(\varrho + u)}{\Eta(\varrho - u)} } = \log(1) = 0$ for $z \in E$.
Finally, we show that the limit
\begin{equation*}
\lim_{z \to \infty} (\log \abs{z} - g_E(z))
= \lim_{u \to \varrho} \left( \log \abs{\varphi(u)} - \log \abs*{ \frac{\Eta(\varrho + u)}{\Eta(\varrho - u)} } \right)
\end{equation*}
exists.  Indeed,
\begin{align*}
\lim_{u\to\varrho} \varphi(u) \frac{\Eta(\varrho-u)}{\Eta(\varrho+u)}
&=\lim_{u\to\varrho} \frac{b_2 \sn^2(u) - b_1 \sn^2(\varrho)}{\Eta(\varrho+u)} 
\frac{\Eta(\varrho-u)}{\sn^2(u) - \sn^2(\varrho)} 
\qquad \text{by~\eqref{eqn:phi}} \\
&= - \frac{(b_2-b_1)\sn^2(\varrho)}{\Eta(2\varrho)} 
\frac{k\,\Theta^4(\varrho)}{\Theta^2(0)\Eta(2\varrho)} 
\qquad\text{by~\eqref{eqn:snsn-Eta}} \\
&= - \frac{(b_4-b_1)(b_3-b_1)}{4(b_2-b_1)} \frac{\Theta^4(0)}{\Theta^4(\varrho)} 
\qquad \text{by~\eqref{eqn:Eta(2u)}~and~\eqref{eqn:cndn_rho}}.
\end{align*}
Altogether, this shows that $g_E$ in~\eqref{eqn:gE} is the Green's function of $\comp{E}$ and completes the proof of~\ref{it:green_function}.
Since $\log(\capacity(E)) = \lim_{z \to \infty} ( \log \abs{z} - g_E(z) )$,
see~\cite[Thm.~5.2.1]{Ransford1995}, we obtain the first identity of~\eqref{eqn:capacity_elliptic}.
The second identity of~\eqref{eqn:capacity_elliptic} follows 
from~\eqref{eqn:sncndn-Theta} and~\eqref{eqn:cndn_rho}.

\ref{it:critical_point_gE}
Since $E = [b_1, b_2] \cup [b_3, b_4]$, the Green's function $g_E$ has exactly one critical point $z_*$, which satisfies $z_* \in \oo{b_2, b_3}$; see~\cite[Thm.~2.8\,(i)]{SchiefermayrSete2023}.
The point $z_*$ is the zero of $\partial_z g_E(z)$.  Substituting $z = \varphi(u)$ and using~\eqref{eqn:wirtinger}, we have
\begin{equation}
2 \partial_z g_E(\varphi(u)) \varphi'(u) = 2 \partial_u ( g_E(\varphi(u)) )
= \frac{\dd}{\dd u} \log \left( \frac{\Eta(\varrho + u)}{\Eta(\varrho - u)} \right).
\end{equation}
By~\eqref{eqn:dlogf} and~\eqref{eqn:phi},
\begin{align*}
\frac{\dd}{\dd u} \log \left( \frac{\Eta(\varrho + u)}{\Eta(\varrho - u)} \right)
&= 2 \zn(\varrho) + \frac{2 \sn(\varrho) \cn(\varrho) \dn(\varrho)}{\sn^2(\varrho) - \sn^2(u)} \\
&= 2 \zn(\varrho) - \frac{2 \cn(\varrho) \dn(\varrho) (\varphi(u) - b_2)}{(b_2 - b_1) \sn(\varrho)}.
\end{align*}
Setting this expression to zero gives
\begin{equation*}
\varphi(u) = b_2 + \frac{(b_2 - b_1) \sn(\varrho) \zn(\varrho)}{\cn(\varrho) \dn(\varrho)}
= b_2 + \sqrt{(b_3 - b_1) (b_4 - b_2)} \zn(\varrho),
\end{equation*}
where the last equality follows from~\eqref{eqn:lambda} and~\eqref{eqn:cndn_rho}.
\end{proof}

\bibliographystyle{siam}
\bibliography{walshmap.bib}

\begin{thebibliography}{10}

\bibitem{Akhiezer1930}
{\sc N.~Achieser}, {\em Sur les polynomes de {Tchebycheff} pour deux
  segments.}, C. R. Acad. Sci., Paris, 191 (1930), pp.~754--756.

\bibitem{Akhiezer1932}
{\sc N.~Achyeser}, {\em {\"U}ber einige {Funktionen}, welche in zwei gegebenen
  {Intervallen} am wenigsten von {Null} abweichen. {I}.}, Bull. Acad. Sci.
  URSS, 1932 (1932), pp.~1163--1202.

\bibitem{AndersonVamanamurthyVuorinen1997}
{\sc G.~D. Anderson, M.~K. Vamanamurthy, and M.~K. Vuorinen}, {\em Conformal
  invariants, inequalities, and quasiconformal maps}, Canadian Mathematical
  Society Series of Monographs and Advanced Texts, John Wiley \& Sons, Inc.,
  New York, 1997.

\bibitem{ByrdFriedman1971}
{\sc P.~F. Byrd and M.~D. Friedman}, {\em Handbook of elliptic integrals for
  engineers and scientists}, vol.~Band 67 of Die Grundlehren der mathematischen
  Wissenschaften, Springer-Verlag, New York-Heidelberg, second~ed., 1971.

\bibitem{CarlsonTodd1983_elliptic}
{\sc B.~C. Carlson and J.~Todd}, {\em The degenerating behavior of elliptic
  functions}, SIAM J. Numer. Anal., 20 (1983), pp.~1120--1129.

\bibitem{ChristiansenSimonZinchenko2025}
{\sc J.~S. Christiansen, B.~Simon, and M.~Zinchenko}, {\em Bounds for weighted
  {C}hebyshev and residual polynomials on subsets of $\mathbb{R}$}, arXiv:
  2502.11424,  (2025).

\bibitem{EremenkoYuditskii2012}
{\sc A.~Eremenko and P.~Yuditskii}, {\em Comb functions}, in Recent advances in
  orthogonal polynomials, special functions, and their applications, vol.~578
  of Contemp. Math., Amer. Math. Soc., Providence, RI, 2012, pp.~99--118.

\bibitem{Fischer1996}
{\sc B.~Fischer}, {\em Polynomial based iteration methods for symmetric linear
  systems}, Wiley-Teubner Series Advances in Numerical Mathematics, John Wiley
  \& Sons, Ltd., Chichester; B. G. Teubner, Stuttgart, 1996.

\bibitem{Grunsky1957a}
{\sc H.~Grunsky}, {\em \"{U}ber konforme {A}bbildungen, die gewisse
  {G}ebietsfunktionen in elementare {F}unktionen transformieren. {I}}, Math.
  Z., 67 (1957), pp.~129--132.

\bibitem{Grunsky1957b}
{\sc H.~Grunsky{}}, {\em \"{U}ber konforme {A}bbildungen, die gewisse
  {G}ebietsfunktionen in elementare {F}unktionen transformieren. {II}}, Math.
  Z., 67 (1957), pp.~223--228.

\bibitem{Grunsky1978}
{\sc H.~Grunsky}, {\em Lectures on theory of functions in multiply connected
  domains}, Vandenhoeck \& Ruprecht, G\"{o}ttingen, 1978.
\newblock Studia Mathematica, Skript 4.

\bibitem{Jenkins1958}
{\sc J.~A. Jenkins}, {\em On a canonical conformal mapping of {J}. {L}.
  {W}alsh}, Trans. Amer. Math. Soc., 88 (1958), pp.~207--213.

\bibitem{Landau1961}
{\sc H.~J. Landau}, {\em On canonical conformal maps of multiply connected
  domains}, Trans. Amer. Math. Soc., 99 (1961), pp.~1--20.

\bibitem{Lawden1989}
{\sc D.~F. Lawden}, {\em Elliptic functions and applications}, vol.~80 of
  Applied Mathematical Sciences, Springer-Verlag, New York, 1989.

\bibitem{LiesenSeteNasser2017}
{\sc J.~Liesen, O.~S\`ete, and M.~M.~S. Nasser}, {\em {Fast and Accurate
  Computation of the Logarithmic Capacity of Compact Sets}}, Comput. Methods
  Funct. Theory, 17 (2017), pp.~689--713.

\bibitem{NasserLiesenSete2016}
{\sc M.~M.~S. Nasser, J.~Liesen, and O.~S\`ete}, {\em Numerical computation of
  the conformal map onto lemniscatic domains}, Comput. Methods Funct. Theory,
  16 (2016), pp.~609--635.

\bibitem{Needham1997}
{\sc T.~Needham}, {\em Visual complex analysis}, The Clarendon Press, Oxford
  University Press, New York, 1997.

\bibitem{Ransford1995}
{\sc T.~Ransford}, {\em Potential theory in the complex plane}, vol.~28 of
  London Mathematical Society Student Texts, Cambridge University Press,
  Cambridge, 1995.

\bibitem{Schiefermayr2008}
{\sc K.~Schiefermayr{}}, {\em An upper bound for the logarithmic capacity of
  two intervals}, Complex Var. Elliptic Equ., 53 (2008), pp.~65--75.

\bibitem{Schiefermayr2011}
{\sc K.~Schiefermayr}, {\em Estimates for the asymptotic convergence factor of
  two intervals}, J. Comput. Appl. Math., 236 (2011), pp.~28--38.

\bibitem{Schiefermayr2015_Zolotarev}
{\sc K.~Schiefermayr}, {\em Zolotarev's conformal mapping and {C}hebotarev's
  problem}, Integral Transforms Spec. Funct., 26 (2015), pp.~118--133.

\bibitem{SchiefermayrSete2023}
{\sc K.~Schiefermayr and O.~S\`ete}, {\em Walsh's {C}onformal {M}ap onto
  {L}emniscatic {D}omains for {P}olynomial {P}re-images {I}}, Comput. Methods
  Funct. Theory, 23 (2023), pp.~489--511.

\bibitem{SchiefermayrSete2024}
{\sc K.~Schiefermayr{} and O.~S\`ete}, {\em Walsh's {C}onformal {M}ap {O}nto
  {L}emniscatic {D}omains for {P}olynomial {P}re-images {II}}, Comput. Methods
  Funct. Theory, 24 (2024), pp.~257--281.

\bibitem{SchiefermayrSete2025}
{\sc K.~Schiefermayr and O.~S\`{e}te}, {\em Walsh's conformal map onto
  lemniscatic domains for several intervals}, Constr. Approx., 62 (2025),
  pp.~565--590.

\bibitem{SeteLiesen2016}
{\sc O.~S\`ete and J.~Liesen}, {\em On conformal maps from multiply connected
  domains onto lemniscatic domains}, Electron. Trans. Numer. Anal., 45 (2016),
  pp.~1--15.

\bibitem{SeteLiesen2017}
{\sc O.~S\`ete and J.~Liesen}, {\em Properties and examples of {F}aber-{W}alsh
  polynomials}, Comput. Methods Funct. Theory, 17 (2017), pp.~151--177.

\bibitem{Suetin1998}
{\sc P.~K. Suetin}, {\em Series of {F}aber polynomials}, vol.~1 of Analytical
  Methods and Special Functions, Gordon and Breach Science Publishers,
  Amsterdam, 1998.

\bibitem{Walsh1956}
{\sc J.~L. Walsh}, {\em On the conformal mapping of multiply connected
  regions}, Trans. Amer. Math. Soc., 82 (1956), pp.~128--146.

\bibitem{Walsh1958}
{\sc J.~L. Walsh{}}, {\em A generalization of {F}aber's polynomials}, Math.
  Ann., 136 (1958), pp.~23--33.

\bibitem{Walsh1969}
{\sc J.~L. Walsh}, {\em Interpolation and approximation by rational functions
  in the complex domain}, Fifth edition. American Mathematical Society
  Colloquium Publications, Vol. XX, American Mathematical Society, Providence,
  R.I., 1969.

\bibitem{Wegert2012}
{\sc E.~Wegert}, {\em Visual complex functions. An introduction with phase
  portraits.}, Birkh\"auser/Springer Basel AG, Basel, 2012.

\bibitem{WegertSemmler2011}
{\sc E.~Wegert and G.~Semmler}, {\em Phase plots of complex functions: a
  journey in illustration}, Notices Amer. Math. Soc., 58 (2011), pp.~768--780.

\bibitem{Widom1969}
{\sc H.~Widom}, {\em Extremal polynomials associated with a system of curves in
  the complex plane}, Advances in Math., 3 (1969), pp.~127--232.

\end{thebibliography}

\end{document}